\documentclass{amsart}

\usepackage{amsmath,amssymb}
\usepackage{mathrsfs}
\usepackage{mathtools}
\usepackage{stmaryrd}
\usepackage{enumerate}
\usepackage{tikz-cd}
\usepackage[all]{xy}
\usepackage{aliascnt}
\usepackage[colorlinks=true, linkcolor=blue, citecolor=blue]{hyperref}

\usepackage{combelow}

\newtheorem{theorem}{Theorem}

\newaliascnt{lemma}{theorem}
\newtheorem{lemma}[lemma]{Lemma}
\aliascntresetthe{lemma}

\newaliascnt{corollary}{theorem}
\newtheorem{corollary}[corollary]{Corollary}
\aliascntresetthe{corollary}

\newaliascnt{proposition}{theorem}
\newtheorem{proposition}[proposition]{Proposition}
\aliascntresetthe{proposition}

\newaliascnt{conjecture}{theorem}
\newtheorem{conjecture}[conjecture]{Conjecture}
\aliascntresetthe{conjecture}

\newaliascnt{question}{theorem}

\aliascntresetthe{question}

\theoremstyle{definition}

\newaliascnt{definition}{theorem}
\newtheorem{definition}[definition]{Definition}
\aliascntresetthe{definition}

\newaliascnt{remark}{theorem}
\newtheorem{remark}[remark]{Remark}
\aliascntresetthe{remark}

\newaliascnt{example}{theorem}
\newtheorem{example}[example]{Example}
\aliascntresetthe{example}

\newaliascnt{notation}{theorem}
\newtheorem{notation}[notation]{Notation}
\aliascntresetthe{notation}

\newtheorem*{acknowledgements}{Acknowledgements}

\newif\ifhascomments \hascommentstrue
\ifhascomments
  \newcommand{\matt}[1]{{\color{red}[[\ensuremath{\spadesuit\spadesuit\spadesuit} #1]]}}
  \newcommand{\jeremy}[1]{{\color{red}[[\ensuremath{\clubsuit\clubsuit\clubsuit} #1]]}}
\else
  \newcommand{\matt}[1]{}
  \newcommand{\jeremy}[1]{}
\fi

\renewcommand{\setminus}{\smallsetminus}

\newcommand{\Z}{\mathbb{Z}}
\newcommand{\bZ}{\mathbb{Z}}

\newcommand{\QQ}{\mathbb{Q}}
\newcommand{\CC}{\mathbb{C}}

\newcommand{\cX}{\mathcal{X}}
\newcommand{\cY}{\mathcal{Y}}
\newcommand{\cZ}{\mathcal{Z}}
\newcommand{\cW}{\mathcal{W}}
\newcommand{\cI}{\mathcal{I}}
\newcommand{\cC}{\mathcal{C}}
\newcommand{\cD}{\mathcal{D}}
\newcommand{\cU}{\mathcal{U}}
\newcommand{\cO}{\mathcal{O}}
\newcommand{\cV}{\mathcal{V}}
\newcommand{\cF}{\mathcal{F}}

\newcommand{\cJ}{\mathscr{J}}
\newcommand{\cK}{\mathcal{K}}
\newcommand{\cE}{\mathcal{E}}

\newcommand{\cL}{\mathcal{L}}

\newcommand{\cR}{\mathcal{R}}

\newcommand{\sL}{\mathscr{L}}
\newcommand{\sM}{\mathscr{M}}

\newcommand{\sJ}{\mathscr{J}}
\newcommand{\sS}{\mathscr{S}}

\newcommand{\bL}{\mathbb{L}}
\newcommand{\bH}{\mathbb{H}}
\newcommand{\bG}{\mathbb{G}}
\newcommand{\bA}{\mathbb{A}}

\newcommand{\diff}{\mathrm{d}}

\newcommand{\id}{\mathrm{id}}
\newcommand{\Gor}{\mathrm{Gor}}

\DeclareMathOperator{\het}{ht}
\DeclareMathOperator{\nhet}{nht}
\DeclareMathOperator{\e}{e}
\DeclareMathOperator{\Spec}{Spec}

\DeclareMathOperator{\ord}{ord}

\DeclareMathOperator{\GL}{GL}

\DeclareMathOperator{\tors}{tors}

\DeclareMathOperator{\coh}{coh}

\DeclareMathOperator{\Fitt}{Fitt}
\DeclareMathOperator{\coker}{coker}
\DeclareMathOperator{\ordjac}{ordjac}
\DeclareMathOperator{\im}{im}
\DeclareMathOperator{\rank}{rank}
\DeclareMathOperator{\diag}{diag}
\DeclareMathOperator{\Div}{div}
\DeclareMathOperator{\SL}{SL}
\DeclareMathOperator{\Gr}{Gr}

\tikzset{cong/.style={draw=none,edge node={node [sloped, allow upside down, auto=false]{$\cong$}}},
         Isom/.style={above,every to/.append style={edge node={node [sloped, allow upside down, auto=false]{$\sim$}}}}}

\title[Crepant resolutions of log-terminal singularities via stacks]{Crepant resolutions of log-terminal singularities via Artin stacks}

\author{Matthew Satriano and Jeremy Usatine}

\thanks{MS was partially supported by a Discovery Grant from the National Science and Engineering Research Council of Canada as well as a Mathematics Faculty Research Chair from the University of Waterloo}

\address{Matthew Satriano, Department of Pure Mathematics, University of Waterloo}
\email{msatriano@uwaterloo.ca}

\address{Jeremy Usatine, Department of Mathematics, Brown University}
\email{jeremy{\_}usatine@brown.edu}

\begin{document}

\begin{abstract}
We prove that every variety with log-terminal singularities admits a crepant resolution by a smooth Artin stack. We additionally prove new McKay correspondences for resolutions by Artin stacks, expressing stringy invariants of $\mathbb{Q}$-Gorenstein varieties in terms of motivic integrals on arc spaces of smooth stacks. In the crepant case, these McKay correspondences are particularly simple, demonstrating one example of the utility of crepant resolutions by Artin stacks.
\end{abstract}

\maketitle

\numberwithin{theorem}{section}
\numberwithin{lemma}{section}
\numberwithin{corollary}{section}
\numberwithin{proposition}{section}
\numberwithin{conjecture}{section}
\numberwithin{question}{section}
\numberwithin{remark}{section}
\numberwithin{definition}{section}
\numberwithin{example}{section}
\numberwithin{notation}{section}

\setcounter{tocdepth}{1}
\tableofcontents



\section{Introduction}

The construction and utilization of crepant resolutions is a theme that underlies myriad subfields of algebraic geometry and mathematical physics. Since their introduction by Reid \cite{ReidCrepant}, they have inspired profound connections between birational geometry, derived categories, representation theory, mirror symmetry, and motivic integration.

The McKay correspondence, conjectured by Reid \cite{Reid} and proved by Batyrev \cite{Batyrev99}, asserts that if $G\subset\SL_n$ is a finite subgroup, $X = \CC^n/G$, and $Y\to X$ is a crepant resolution, then the Euler characteristic of $Y$ is equal to the number of conjugacy classes in $G$. 
Building on work of Kapranov--Vasserot \cite{KV}, Bridgeland--King--Reid showed in their groundbreaking paper \cite{BKR} that the McKay Correspondence for threefolds lifts to an equivalence of derived categories. 
This led Bondal and Orlov \cite{BO} to conjecture that all crepant resolutions of Gorenstein singularities have equivalent derived categories. 
This conjectural derived McKay correspondence has spurred fruitful connections between derived categories and the minimal model program \cite{BrideglandFlop,BridgelandStability}, geometric invariant theory \cite{CI,HalpernLeistner,BallardFaveroKatzarkov}, 
and non-commutative geometry \cite{vdB2004,vdB10,IW13,SpenkoVanDenBergh17}. Versions of the McKay correspondence have also had great influence in Gromov--Witten theory, Donaldson--Thomas theory, and mirror symmetry. For varieties with quotient singularities, the Crepant Resolution Conjecture \cite{Ruan06} compares the Chen--Ruan orbifold cohomology \cite{ChenRuan} to the Gromov--Witten theory of a crepant resolution when it exists. This has led to a flurry of activity in the subject \cite{BG,CCIT,CCITtoric,CIJ} as well as versions of the conjecture comparing Donaldson--Thomas invariants \cite{BCY,DTCRC}.


Of particular interest in this paper are the connections between crepant resolutions and motivic integration. The theory of motivic integration was pioneered by Kontesevich \cite{Kontsevich} in his proof that birational smooth Calabi--Yau varieties have equal Hodge numbers. The field has since had broad ranging applications in birational geometry, mirror symmetry, and the study of singularities. Central to these applications is a \emph{motivic change of variables formula} relating the motivic measures of $Y$ and $X$ under a birational modification $X \to Y$, see e.g., \cite{DenefLoeser1999, Looijenga}. Motivated by mirror symmetry for singular Calabi-Yau varieties, Batyrev \cite{Batyrev1998} introduced the \emph{stringy Hodge numbers} for varieties with \emph{log-terminal singularities}; these are defined in terms of the combinatorial information of a resolution of singularities rather than the dimensions of cohomology groups. Whenever $Y$ admits a crepant resolution $X\to Y$, 
the stringy Hodge numbers of $Y$ agree with their classical counterparts on $X$.
 In \cite{DenefLoeser2002}, Denef and Loeser refined the stringy Hodge numbers by introducing the \emph{Gorenstein measure} $\mu^\Gor_Y$ on the arc scheme $\sL(Y)$; this measure assumes values in a modified Grothendieck ring of varieties and recovers the stringy Hodge numbers after a suitable specialization. By proving a motivic change of variables formula for $\mu^\Gor_Y$, Denef and Loeser obtained a motivic version of the McKay correspondence, refining the original version proved by Batyrev \cite{Batyrev99}.

In order to study stringy Hodge numbers of $\QQ$-Gorenstein quotient singularities, Yasuda in his beautiful paper \cite{Yasuda2004} (see also \cite{Yasuda2006, Yasuda2019}) considered crepant resolutions by smooth Deligne--Mumford stacks. Specifically, given $Y$ with $\QQ$-Gorenstein quotient singularities, Vistoli \cite{Vistoli89} constructed a canonical small (hence crepant) resolution $\pi\colon\cX\to Y$ with $\cX$ smooth Deligne--Mumford. By proving a motivic change of variables formula for $\pi$, Yasuda expressed $\mu^\Gor_Y$ as a certain motivic integral over arcs of $\cX$, from which he deduced that the stringy Hodge numbers of $Y$ agree with the orbifold Hodge numbers of $\cX$.

Outside the context of $\QQ$-Gorenstein quotient singularities, stringy Hodge numbers are generally not well understood. Since varieties with log-terminal singularities 
do not usually admit crepant resolutions by Deligne--Mumford stacks, if one wishes to glean information about such varieties through stack-theoretic methods, it is necessary to consider Artin stacks. Indeed, Example \ref{ex:noDM-or-NCCR-res} gives an explicit log-terminal variety that has a small resolution by an Artin stack, yet has no crepant resolution by a Deligne--Mumford stack and no non-commutative crepant resolution (NCCR).

One is therefore in need of more general existence theorems for crepant resolutions and more general versions of the motivic change of variables formula in the context of smooth Artin stacks. The use of smooth Artin stacks in the study of singularities has already had great success in several areas of algebraic geometry:~Halpern--Leistner \cite{HalpernLeistner} and Ballard--Favero--Katzarkov \cite{BallardFaveroKatzarkov} made significant progress on the derived McKay correspondence by comparing derived categories of varieties with an ambient Artin stack coming from a variation of GIT; Artin stacks have played an essential role in recent functorial resolution of singularities algorithms \cite{Quek,AbramovichQuekLogRes} building on previous stack-theoretic resolution algorithms of \cite{AbramovichTemkinWlodarczyk,AbramovichTemkinWlodarczyk2, McQuillanMarzo}; finally, Artin stacks have been impressively applied by Quek \cite{Quekmonodromy} in the study of the monodromy conjecture.




\vspace{0.5em}

In this paper, we:
\begin{enumerate}
\item Prove that all log-terminal singularities have crepant resolutions by smooth Artin stacks (Theorem \ref{thm:crepant-log-terminal}).
\item Prove a motivic change of variables formula for resolutions by smooth Artin stacks (Theorem \ref{thm:main-mcvf}). This formula involves the relative canonical divisor of the resolution and is therefore particularly simple in the crepant case.
\end{enumerate}
Thus, crepant resolutions exist in the largest possible generality one needs from the perspective of motivic McKay correspondence. (See Subsection \ref{subsec:statement-results} for precise definitions.)

\begin{theorem}\label{thm:crepant-log-terminal}
Let $Y$ be an irreducible variety with log-terminal singularities over an algebraically closed field $k$ of characteristic $0$. Then there exists a smooth finite type Artin stack $\cX$ over $k$ with affine diagonal and a crepant resolution $\pi\colon\cX\to Y$ which is an isomorphism over the smooth locus $Y^{\mathrm{sm}}$.
\end{theorem}

Interestingly, the proof of the existence of this resolution relies on our motivic change of variables formula (Theorem \ref{thm:main-mcvf}). We construct $\cX$ via a moduli space parameterizing degenerations of framed bundles. To prove $\cX$ is crepant, we compute the Gorenstein measure $\mu_Y^{\Gor}$ and the motivic measure $\mu_\cX$ on appropriately chosen cylinders; since the relative canonical divisor $K_{\cX/Y}$ arises as the natural discrepancy term in our change of variables formula, by comparing the values of the measures, we are able to conclude $K_{\cX/Y}=0$.

Theorem \ref{thm:main-mcvf} builds on our previous work \cite[Theorem 1.3]{SatrianoUsatine2}, where we proved a general motivic change of variables formula for Artin stacks relating $\mu_\cX$ and $\mu_Y$ (as opposed to $\mu_Y^{\Gor}$). This made use of our previous work \cite{SatrianoUsatine} where we established such a formula for toric stacks as well as laid the foundation for motivic integration on Artin stacks. 


For applications to stringy invariants, one is most interested in replacing $\mu_Y$ with $\mu_Y^{\Gor}$. Ideally, one would hope to prove a motivic change of variables formula of the form
\begin{equation}\label{eqn:intro-mcvf}
\mu_Y^{\Gor}(D) = \int_\cC \bL^{(-1/m)\ord_{mK_{\cX/Y}}} \diff\mu_\cX
\end{equation}
when $Y$ is $m$-Gorenstein 
and $D$ and $\cC$ are appropriate cylinders. Several technical hurdles arise here. The first is a minor one:~there is no established definition of a canonical divisor $K_\cX$ for Artin stacks, so one does not \emph{a priori} know what the right hand side of \eqref{eqn:intro-mcvf} should be; 
the definition we take is the obvious one if one defines the canonical sheaf $\omega_\cX$ to be the determinant of the cotangent complex $\det(L_\cX)$. Second, even in the simplest case where $\pi\colon\cX\to Y$ is a small resolution (and $mK_{\cX/Y}=0$ regardless of what definition of $mK_{\cX/Y}$ one takes), it is not at all clear that \eqref{eqn:intro-mcvf} holds, i.e., it is not obvious that $\mu_Y^{\Gor}(D)=\mu_\cX(\cC)$. 
The reason for this is that in the motivic change of variables formula \cite[Theorem 1.3]{SatrianoUsatine2} relating $\mu_Y$ and $\mu_\cX$, the natural correction factor that arises is given by a difference of height functions of certain cotangent complexes.\footnote{The definition of height function we introduced was inspired by \cite{ESZB} where the number-theoretic Weil height functions were extended to the case of vector bundles on stacks.} Unlike the scheme case, this difference is controlled by the hypercohomologies of \emph{complexes} of sheaves (as opposed to a single sheaf), and it is no longer apparent how to relate this difference to an invertible ideal sheaf.

In Theorem \ref{thm:main-mcvf}, we prove the motivic change of variables \eqref{eqn:intro-mcvf} and, in fact, show that there exists a \emph{unique} divisor $mK_{\cX/Y}$ for which \eqref{eqn:intro-mcvf} holds. 
This therefore allows us to express the stringy invariants of $Y$ as a motivic integral over arcs of $\cX$. In particular, this proves \cite[Conjecture 1.1]{SatrianoUsatine} when $\cC\to C$ bijective. This is part of our ongoing program to express the stringy Hodge numbers of $Y$ as dimensions of a Chen--Ruan style cohomology theory on $\cX$.

It is somewhat surprising that the motivic change of variables formula \eqref{eqn:intro-mcvf} takes such a simple form. Indeed, our stack-theoretic motivic change of variables for $\mu_Y$ (\cite[Theorem 1.3]{SatrianoUsatine2}) is more complicated than its scheme-theoretic counterpart due to the presence of the non-vanishing height function $\het^{(1)}_{L_{\cX/Y}}$. Given that the change of variables formula for $\mu_Y^{\Gor}$ relies on the one for $\mu_Y$, one may have expected a formula whose right hand side has several terms in the exponent of $\bL$. 



Finally, given the ubiquity of crepant resolutions, we hope Theorem \ref{thm:crepant-log-terminal} will be useful beyond the scope of motivic integration.





\subsection{Statement of further results} \label{subsec:statement-results}

Throughout this paper, $k$ is an algebraically closed field of characteristic 0. Ultimately, we are interested in resolving a singular variety $Y$ by a smooth Artin stack $\cX$. Our results will apply for resolutions $\pi\colon\cX\to Y$ which obey a weak form of birationality that we now introduce.

\begin{definition}
We say a morphism $\cX \to \cY$ of Artin stacks is \emph{weakly birational} if there exists a dense open substack $\cU \hookrightarrow \cX$ such that the composition $\cU \hookrightarrow \cX \to \cY$ is an open immersion with dense image.
\end{definition}

We also define a stronger version of birationality.

\begin{definition}
A morphism $\cX \to \cY$ of Artin stacks is \emph{strongly birational} if there exists a dense open substack $\cV \hookrightarrow \cY$ such that $\cV \times_\cY \cX \hookrightarrow \cX$ has dense image and $\cV \times_\cY \cX \to \cV$ is an isomorphism.
\end{definition}

\begin{remark}
An example where these two definitions differ is the good moduli space map $[\bA^1_k / \bG_{m,k}] \to \Spec(k)$; this is weakly birational but not strongly birational.
\end{remark}

Throughout the rest of the introduction, we fix the following notation.

\begin{notation}\label{running-notation}
Let $\cX$ be a smooth irreducible finite type Artin stack over $k$, let $Y$ be 
a finite type scheme over $k$, let $m \in \Z_{>0}$ be such that $Y$ is $m$-Gorenstein, and let $\pi: \cX \to Y$ be a weakly birational morphism. Fix a non-empty open substack $\cU \hookrightarrow \cX$ such that $\cU \hookrightarrow \cX \to Y$ is an open immersion and let $\eta: \Spec k(\cX) \to \cX$ be the generic point of $\cX$, see Definition \ref{def:generic-pt-stack}.
\end{notation}

Since $\cX$ is smooth, the cotangent complex $L_\cX:=L_{\cX/k}$, as defined in \cite{Olsson2007}, is perfect; hence its determinant $\det(L_\cX)$ is a well-defined invertible sheaf, which we denote by $\omega_\cX$. We may now introduce our definition of the relative canonical divisor.

\begin{definition}\label{def:rel-canonical-divisor}
The $m$th \emph{relative canonical sheaf} $\omega_{\cX/Y, m}$ 
is defined by
\[
	\omega_{\cX/Y, m} = \omega_\cX^{\otimes m} \otimes \pi^*\omega_{Y,m}^\vee.
\]
The $m$th \emph{relative canonical divisor} $mK_{\cX/Y}$ is the Cartier divisor associated to the canonical isomorphism $\eta^*\omega_{\cX/Y, m} \xrightarrow{\sim} \cO_{k(\cX)}$, see Proposition \ref{prop:rat-sec-->Cartier}.
\end{definition}

\begin{remark}\label{rmk:mKXY-support}
Weak birationality of $\pi$ ensures that $\omega_{\cX/Y, m}$ is canonically trivialized 
when pulled back by $\eta$.  Furthermore, Proposition \ref{prop:rat-sec-->Cartier} tells us that $mK_{\cX/Y}$ satisfies:
\begin{itemize}

\item $\cO(mK_{\cX/Y}) \cong \omega_{\cX/Y, m}$.

\item 
$mK_{\cX/Y}$ is supported on the complement of $\cU$.
\end{itemize}
\end{remark}

With our definition of $mK_{\cX/Y}$, we may now extend the notion of crepant resolutions to the case of Artin stacks.

\begin{definition}\label{def:crepant}
The map $\pi\colon\cX\to Y$ is \emph{crepant} if $mK_{\cX/Y}=0$. 
We say $\pi$ is a \emph{crepant resolution} if $\pi$ is crepant, strongly birational, and factors as $\cX\xrightarrow{\pi_1}\cY\xrightarrow{\pi_2} Y$ where $\pi_1$ is a good moduli space map and $\pi_2$ is proper. (Recall from Notation \ref{running-notation} that $\cX$ is smooth.)
\end{definition}

\begin{remark}\label{rmk:crepant-distinction}
An important distinction to note is that triviality of $mK_{\cX/Y}$ is not equivalent to triviality of $\omega_{\cX/Y}$ even when $\cX$ is a variety (and $\pi$ is not proper), see Example \ref{ex:crepant-distinction}.
\end{remark}

Given an effective Cartier divisor $\cD$ on $\cX$, we define its \emph{order function} to be the order function of its ideal sheaf $\ord_\cD:=\ord_{\cO(-\cD)}$. For arbitrary Cartier divisors $\cD=\cD_+-\cD_-$ with $\cD_\pm$ effective, we define $\ord_\cD:=\ord_{\cD_+}-\ord_{\cD_-}$.

We may now state our motivic change of variables formula, see \S\ref{conventions} for further notation.

\begin{theorem}\label{thm:main-mcvf}
Using Notation \ref{running-notation}, if $\cX$ has affine geometric stabilizers and separated diagonal, then
\begin{enumerate}[(a)]

\item For any cylinder $\cC \subset |\sL(\cX)| \setminus |\sL(\cX \setminus \cU)| \subset |\sL(\cX)|$, the function
\[
	\ord_{mK_{\cX/Y}}: \cC \to \Z
\]
is constructible.\vspace{0.4em}

\item\label{mcvf-main-part} If $\cC \subset |\sL(\cX)| \setminus |\sL(\cX \setminus \cU)| \subset |\sL(\cX)|$ and $D \subset \sL(Y)$ are cylinders such that $\overline{\cC}(k') \to D(k')$ is a bijection for all field extensions $k'$ of $k$, then
\[
	\mu_Y^{\Gor}(D) = \int_{\cC} \bL^{(-1/m)\ord_{mK_{\cX/Y}}} \diff\mu_\cX \in \widehat{\sM}_k[\bL^{1/m}].
\]
\end{enumerate}
\end{theorem}

Since small resolutions are crepant by Remark \ref{rmk:mKXY-support}, we obtain the following immediate corollary of Theorem \ref{thm:main-mcvf}.

\begin{corollary}\label{cor:main-crepant}
With hypotheses as in Theorem \ref{thm:main-mcvf}, if $\pi$ is crepant, then 
\[
\mu_Y^{\Gor}(D) = \mu_{\cX}(\cC).
\]
In particular, this holds if $\pi$ is a small resolution.
\end{corollary}

\begin{remark}
By Theorem \ref{thm:crepant-log-terminal}, Corollary \ref{cor:main-crepant} is applicable to every variety with log-terminal singularities.
\end{remark}


The main difficulty in proving Theorem \ref{thm:main-mcvf} is showing that $\ord_{mK_{\cX/Y}}$ differs from the correction factor in \cite[Theorem 1.3]{SatrianoUsatine2} precisely by the order function of the ideal $\cJ_{Y,m}$, the unique ideal sheaf such that the image of $(\Omega_Y^{\dim Y})^{\otimes m}\to\omega_{Y,m}$ is given by $\cJ_{Y,m}\omega_{Y,m}$. After a minor reformulation of \cite[Theorem 1.3]{SatrianoUsatine2} given in Theorem \ref{theoremChangeOfVariables-reformulation}, this amounts to showing the following formula. 
See Definition 3.1 of (loc.~cit) for the definition of $\het^{(i)}_{L_{\cX/Y}}$.

\begin{theorem}\label{thm:main-canonical-divisor}
Using Notation \ref{running-notation},
\[
	\ord_{mK_{\cX/Y}} = m \het^{(0)}_{L_{\cX/Y}} - m \het^{(1)}_{L_{\cX/Y}} -\ord_{\sJ_{Y,m}} \circ \sL(\pi).
\]
on $|\sL(\cX)| \setminus |\sL(\cX \setminus \cU)|$.
\end{theorem}


One of the key ingredients in the proof of Theorem \ref{thm:main-canonical-divisor} is constructing Fitting ideals $\cJ_i$ which read off the height functions of $L_{\cX/Y}$. It is worth mentioning that for any complex $\cE$, one has easiest access to the maximal non-vanishing height function $\het^{(i)}_\cE$. When $\cX=X$ is a variety, 
$\het^{(1)}_{L_{X/Y}}=0$ and the maximal non-vanishing height function $\het^{(0)}_{L_{X/Y}}$ can be reinterpreted as the order function associated to $\Fitt_0(\Omega^1_{X/Y})$. In the case of stacks, however, $\het^{(1)}_{L_{X/Y}}$ no longer vanishes
and the computations required to relate $mK_{\cX/Y}$, $\cJ_0$, $\cJ_1$, and $\cJ_{Y,m}$ become significantly more involved. 

Theorem \ref{thm:main-canonical-divisor} allows us to give an alternative characterization of crepant morphisms and answer \cite[Question 1.8]{SatrianoUsatine2}, where we asked:~under the hypotheses of Theorem \ref{thm:main-mcvf}, if $Y$ is Gorenstein and $\pi$ is a small resolution which is \emph{strongly} birational, then do we have
\begin{equation}\tag{$\triangle$}\label{eqn:condition-triangle}
\het^{(0)}_{L_{\cX/Y}} - \het^{(1)}_{L_{\cX/Y}} -\frac{1}{m}\ord_{\sJ_{Y,m}} \circ \sL(\pi)=0\quad\textrm{on}\quad |\sL(\cX)|\setminus|\sL(\cX\setminus\cV)|
\end{equation}
for some open substack $\cV\hookrightarrow\cX$ with $\cV\hookrightarrow\cX\to Y$ an open immersion? In fact, $(\triangle)$ was initially proposed as a definition of crepantness for strongly birational maps. We may now answer this question by proving the following stronger criterion.

\begin{corollary}\label{cor:triangle-iff-crepant}
With hypotheses as in Theorem \ref{thm:main-mcvf}, $\pi$ is crepant if and only if \eqref{eqn:condition-triangle} holds. 
In particular, \cite[Question 1.8]{SatrianoUsatine2} has an affirmative answer.
\end{corollary}

Corollary \ref{cor:triangle-iff-crepant} is an application of Proposition \ref{prop:order-fnc-determines-div}, which proves that Cartier divisors on stacks are determined by their order functions.

It is worth highlighting the case where $Y$ is lci; this has the particularly appealing aspect that $\het^{(-1)}_{L_{\cX/Y}}=\ord_{\cJ_Y}$, see Proposition \ref{prop:Ylci->perfect-ht-2-vanish}. As a result, the right hand side of Theorem \ref{thm:main-canonical-divisor} can be reinterpreted as an alternating sum of height functions of $L_{\cX/Y}$. Furthermore, in this case $L_{\cX/Y}$ is perfect and so has a well-defined determinant with a canonical trivialization $\eta^*\det(L_{\cX/Y})\simeq\cO_{k(\cX)}$. This yields a Cartier divisor $\Div(L_{\cX/Y})$ as we show in Proposition \ref{prop:alternating-sum-hts-ord}. This Cartier divisor in fact agrees with $K_{\cX/Y}$.

\begin{theorem}\label{thm:main-lci-case}
If $Y$ is lci, then $K_{\cX/Y}=\Div(L_{\cX/Y})$ and hence $\omega_{\cX/Y}\simeq\det(L_{\cX/Y})$. Furthermore, 
\[
	\ord_{K_{\cX/Y}} = \sum_{-1\leq i\leq 1}(-1)^i\het^{(i)}_{L_{\cX/Y}}.
\]
\end{theorem}

Lastly, we believe our existence theorem for crepant resolutions (Theorem \ref{thm:crepant-log-terminal}) is optimal, as we now explain. For a $\QQ$-Gorenstein variety $Y$, the integral defining $\mu_Y^{\Gor}(\sL(Y))$ converges if and only if $Y$ has log-terminal singularities. In light of our motivic change of variables formula (Theorem \ref{thm:main-mcvf}), the following conjecture is therefore reasonable.

\begin{conjecture}\label{conj:crepant-->log-terminal}
If $Y$ is an irreducible $\QQ$-Gorenstein variety, $\pi\colon\cX\to Y$ is a crepant resolution, and $\cX$ has affine diagonal, then $Y$ has log-terminal singularities.
\end{conjecture}

To verify Conjecture \ref{conj:crepant-->log-terminal}, one would need a careful understanding of which arcs of $Y$ admit lifts to arcs of $\cX$. This will be the subject of future work, where we aim to expand the scope of motivic integration, allowing for canonical lifts of arcs.


\subsection{Conventions}\label{conventions}
For any stack $\cX$ over $k$, we let $|\cX|$ denote its associated topological space, and for any subset $\cC \subset |\cX|$ and field extension $k'$ of $k$, we let $\cC(k')$ (resp. $\overline{\cC}(k')$) denote the category of (resp. set of isomorphism classes of) $k'$-valued points of $\cX$ whose class in $|\cX|$ is contained in $\cC$.

\begin{acknowledgements}
This project benefited greatly from conversations with many people. We are grateful to Jason Bell, Bhargav Bhatt, Dan Edidin, Patricia Klein, Lucia Martin Merchan, Mircea Musta\cb{t}\u{a}, Martin Olsson, Karl Schwede, Karen Smith, Michel Van den Bergh, and Takehito Yasuda.
\end{acknowledgements}

\section{Reformulating our motivic change of variables formula}\label{sec:proof-of-conjecture}

Our motivic change of variables formula \cite[Theorem 1.3]{SatrianoUsatine2} involves height functions of different complexes:~$\het^{(0)}_{L\sigma^*L_{\cI\cX/\cX}} - \het^{(0)}_{L_{\cX/Y}}$, where $I\cX$ is the inertia stack and $\sigma\colon\cX\to I\cX$ is the identity section. We begin by rewriting this purely in terms of $L_{\cX/Y}$.

\begin{lemma}\label{l:het0IX-LXY}
If $\cX$ is a finite type Artin stack over a base scheme $S$, 
then there is a natural quasi-isomorphism $L\sigma^*L_{\cI\cX/\cX}\simeq L_{\cX/S}[1]$.
\end{lemma}
\begin{proof}
From the diagram
\[
\xymatrix{
\cX\ar[r]^-{\sigma}\ar[dr]_-{\id} \ar@/^1.5pc/[rr]^{\id} & I\cX\ar[r]^-{p}\ar[d] & \cX\ar[d]^-{\Delta} \\
& \cX\ar[r]^-{\Delta} & \cX\times_S \cX
}
\]
with the square being cartesian, we have
\[
L\sigma^*L_{I\cX/\cX}=L\sigma^*Lp^*L_{\Delta} = L_{\Delta}.
\]
From the diagram
\[
\xymatrix{
\cX\ar[r]^-{\Delta}\ar[dr]_-{\id}\ar@/^1.5pc/[rr]^{\id} & \cX\times_S\cX\ar[r]^-{q}\ar[d] & \cX\ar[d] \\
& \cX\ar[r] & S
}
\]
with the square being cartesian, we have
\[
L_\Delta=L\Delta^*L_{\cX\times_S\cX/\cX}[1]=L\Delta^*Lq^*L_{\cX/S}[1]=L_{\cX/S}[1].
\]
Thus, $L\sigma^*L_{I\cX/\cX}=L_{\cX/S}[1]$.
\end{proof}


\begin{corollary}\label{cor:ht1-ht0-LXY}
Using Notation \ref{running-notation},
\[
\het^{(0)}_{L\sigma^*L_{\cI\cX/\cX}} 
=\het^{(1)}_{L_{\cX/Y}}.
\]
\end{corollary}
\begin{proof}
By Lemma \ref{l:het0IX-LXY}, we have $\het^{(0)}_{L\sigma^*L_{\cI\cX/\cX}}=\het^{(1)}_{L_\cX}$. Consider the exact triangle
\[
L\pi^*L_Y\to L_\cX\to L_{\cX/Y}.
\]
Let $D=\Spec k'[[t]]$ with $k'/k$ a field extension, let $\varphi\colon D\to\cX$ be an arc, and let $\psi=\pi\circ\varphi$. Applying $L\varphi^*$ to the exact triangle above yields an exact triangle
\[
L\psi^*L_Y\to L\varphi^*L_\cX\to L\varphi^*L_{\cX/Y}.
\]
Since $L\psi^*L_Y$ is concentrated in non-positive degrees, we see
\[
L^1\varphi^*L_\cX\simeq L^1\varphi^*L_{\cX/Y},
\]
and hence $\het^{(1)}_{L_{\cX}}=\het^{(1)}_{L_{\cX/Y}}$.
\end{proof}

In light of Corollary \ref{cor:ht1-ht0-LXY}, we may rewrite the motivic change of variables formula \cite[Theorem 1.3]{SatrianoUsatine2}; we may also remove the unnecessary hypothesis $\dim\cX=\dim Y$, which is automatically satisfied given the existence of $\cU$.

\begin{theorem}\label{theoremChangeOfVariables-reformulation}
With hypotheses as in Theorem \ref{thm:main-mcvf},
\begin{enumerate}[(a)]
\item For any cylinder $\cC \subset |\sL(\cX)| \setminus |\sL(\cX \setminus \cU)| \subset |\sL(\cX)|$, the function
\[
	\het^{(0)}_{L_{\cX/Y}} - \het^{(1)}_{L_{\cX/Y}} : \cC \to \Z
\]
is constructible.\vspace{0.4em}

\item\label{theoremPartChangeOfVariables} If $\cC \subset |\sL(\cX)| \setminus |\sL(\cX \setminus \cU)| \subset |\sL(\cX)|$ and $D \subset \sL(Y)$ are cylinders such that $\overline{\cC}(k') \to D(k')$ is a bijection for all field extensions $k'$ of $k$, then
\[
	\mu_Y(D) = \int_{\cC}\bL^{\het^{(0)}_{L_{\cX/Y}} - \het^{(1)}_{L_{\cX/Y}}} \diff \mu_\cX.
\]

\end{enumerate}
\end{theorem}

\section{Computing heights using Fitting ideals}\label{sec:hts-fitting}

Throughout this section, we use Notation \ref{running-notation}. We additionally fix a smooth cover $\rho\colon\widetilde{X}\to\cX$, let
\[
d=\dim Y=\dim\cX\quad\textrm{and}\quad r=\dim\widetilde{X}-\dim\cX.
\]
Consider the diagram
\begin{equation}\label{eqn:main-diag}
\xymatrix{
\widetilde{U}\ar@{^(->}[r]^-{\widetilde{\jmath}}\ar[d]_-p & \widetilde{X}\ar[d]^-\rho\\
\cU\ar@{^(->}[r]^-{j}\ar@^{^(->}[rd]_-i & \cX\ar[d]^-\pi\\
& Y
}
\end{equation}
where the square is cartesian. Then we have an exact triangle
\begin{equation}\label{eq:rhoL_X}
\rho^*L_{\cX}\to \Omega^1_{\widetilde{X}}\to \Omega^1_{\widetilde{X}/\cX}.
\end{equation}
For schemes, the relative Jacobian ideal plays an important role in motivic integration; this is given by the $0$th Fitting ideal of the relative differentials and it reads off the $0$th height function. In the case of stacks, we must gain access to both the $0$th and $1$st height functions of $L_{\cX/Y}$. We accomplish this by considering Fitting ideals coming from the maps in the following $3$-term complex.

\begin{definition}\label{def:Fitt-E}
Let
\[
\cE_{\widetilde{X}/\cX/Y}:=[\rho^*\pi^*\Omega^1_Y\xrightarrow{d_0} \Omega^1_{\widetilde{X}}\xrightarrow{d_1} \Omega^1_{\widetilde{X}/\cX}]
\]
with $\Omega^1_{\widetilde{X}/\cX}$ is degree $1$. Define the ideal sheaves
\[
\cJ_1:=\Fitt_0(\coker(d_1))\quad\textrm{and}\quad \cJ_0:=\Fitt_r(\coker(d_0)).
\]
\end{definition}

\begin{remark}\label{rmk:Fitt-E-not-surj}
Unlike the case of schemes, $d_1$ need not be surjective and so $\cJ_1$ can be a non-trivial ideal. Since $\coker(d_0)=\Omega^1_{\widetilde{X}/Y}$, we see
\[
\cJ_0=\Fitt_r(\Omega^1_{\widetilde{X}/Y}).
\]
In the case where $\cX=\widetilde{X}$ is a smooth scheme, this agrees with the relative Jacobian ideal.
\end{remark}

\begin{remark}\label{rmk:Fitt-E}
By definition, $\cJ_0$ and $\cJ_1$ are the images of the maps
\[
\rho^*\pi^*\Omega^d_Y\otimes(\Omega^d_{\widetilde{X}})^\vee\twoheadrightarrow \cJ_0\subset\cO_{\widetilde{X}}
\]
\[
\Omega^r_{\widetilde{X}}\otimes(\Omega^r_{\widetilde{X}/\cX})^\vee\twoheadrightarrow \cJ_1\subset\cO_{\widetilde{X}}
\]
induced by $d_0$ and $d_1$, see e.g.,~\cite[p 492--493]{EisenbudCA}. 
\end{remark}

The following result shows the utility of $\cJ_0$ and $\cJ_1$. When $\cX=X$ is a scheme, this recovers the fact that $\ordjac(\varphi)$ is computed by $\Fitt_0(\Omega^1_{X/Y})$.

\begin{proposition}\label{prop:ordJi-computes-hts}
Let $D=\Spec k'[[t]]$ with $k'/k$ a field extension, let $\widetilde{\varphi}\colon D\to\widetilde{X}$ be an arc, and let $\varphi=\rho\circ\widetilde{\varphi}$. Then
\[
\ord_{\cJ_i}(\widetilde{\varphi})=\het^{(i)}_{L_{\cX/Y}}(\varphi)
\]
for $i\in\{0,1\}$.
\end{proposition}

The rest of this section is devoted to the proof of Proposition \ref{prop:ordJi-computes-hts}. We begin with a general result.

\begin{lemma}\label{l:Fitting-torsion-coh-length}
Let $R=k'[[t]]$ with $k'$ a field and consider a complex
\[
M\xrightarrow{\alpha} R^{a+b}\xrightarrow{\beta} R^b
\]
of finitely generated $R$-modules, where $M$ has rank $a$. Suppose that the cohomology modules $Q=\ker(\beta)/\im(\alpha)$ and $\coker(\beta)$ are torsion. Then
\[
\dim_{k'}Q=\ord_t\Fitt_{b-a}(\coker(\alpha))
\]
and
\[
\dim_{k'}\coker(\beta)=\ord_t\Fitt_0(\coker(\beta)).
\]
\end{lemma}
\begin{proof}
First, $\im(\alpha)$ and $\ker(\beta)$ are submodules of the free module $R^{a+b}$, hence free. We see $\ker(\beta)$ has rank at least $a$. On the other hand, since $\rank(M)=a$, we see $\im\alpha$ has rank at most $a$. Since $Q$ is torsion, we must have
\[
\rank\im(\alpha)=\rank\ker(\beta)=a.
\]
By the structure theorem for modules over a PID, we may choose bases so that the map
\[
R^a\simeq\im(\alpha)\to\ker(\beta)\simeq R^a
\]
is given by a diagonal matrix $\diag(t^{e_1},\dots,t^{e_a})$; hence, $\dim_{k'} Q=\sum e_i$.

Next, since $\im\beta$ is a submodule of the free module $R^b$, it is also free. We may therefore choose a splitting of the short exact sequence
\[
0\to\ker(\beta)\to R^b\to \im(\beta)\to 0.
\]
Hence the map $\alpha$ factors as
\[
M\twoheadrightarrow \im(\alpha)\simeq R^a\xrightarrow{\diag(t^{e_1},\dots,t^{e_a})} R^a\simeq \ker(\beta)\subset \ker(\beta)\oplus\im(\beta)\simeq R^b.
\]
We therefore have a presentation
\[
R^a\xrightarrow{P} R^b\to \coker(\alpha)\to 0
\]
where $P$ is the matrix whose upper $a\times a$ block is $\diag(t^{e_1},\dots,t^{e_a})$ and all other entries are $0$. By definition, we then have
\[
\Fitt_{b-a}(\coker(\alpha))=(t^{\sum e_i}),
\]
which gives the desired formula for $\dim_{k'}Q$.

The equality
\[
\dim_{k'}\coker(\beta)=\ord_t\Fitt_0(\coker(\beta)),
\]
is proved in an analogous fashion. Observe that since $\coker(\beta)$ is torsion, we must have $\im\beta\simeq R^b$. By another application of the structure theorem for modules over a PID we may therefore assume $\beta$ is given by a matrix whose first $b\times b$ block is $\diag(t^{e'_1},\dots,t^{e'_b})$ and all other entries are $0$. Then $\dim_{k'}\coker(\beta)=\sum e'_i$ and $\Fitt_0(\coker(\beta))=(t^{\sum e'_i})$, proving the result.
\end{proof}

\begin{definition}\label{def:naiveht}
Let $\cY$ be a $k$-stack and let $\cF$ be a complex of coherent sheaves on $\cY$. For any field extension $k'/k$ and any arc $\gamma\colon\Spec k'[[t]]\to\cY$, define the \emph{naive height} by
\[
\nhet^{(i)}_\cF(\gamma):=\dim_{k'}\bH^i(\gamma^*\cF),
\]
where $\gamma^*$ denotes usual (not derived) pullback.
\end{definition}
\begin{remark}
The naive height differs from the height introduced in \cite[Definition 3.1]{SatrianoUsatine2} in two ways. First, $\cF$ is a complex of sheaves as opposed to an object of the derived category. Second, as pointed out above, $\gamma^*$ is not the derived pullback.
\end{remark}

Next, we show that the naive height of $\cE_{\widetilde{X}/\cX/Y}$ computes the usual heights of $L_{\cX/Y}$ in the degrees of interest.

\begin{proposition}\label{prop:E-LXY-same-ht}
Let $D=\Spec k'[[t]]$, let $\widetilde{\varphi}\colon D\to\widetilde{X}$ be an arc, and let $\varphi=\rho\circ\widetilde{\varphi}$. Then for $i\in\{0,1\}$, we have
\[
\het^{(i)}_{L_{\cX/Y}}(\varphi)=\nhet^{(i)}_{\cE_{\widetilde{X}/\cX/Y}}(\widetilde{\varphi}).
\]
\end{proposition}
\begin{proof}
Let $\psi=\pi\circ\varphi$. Applying $\rho^*$ to the exact triangle
\[
L\pi^*L_Y\to L_\cX\to L_{\cX/Y},
\]
we obtain an exact triangle
\[
\rho^*L\pi^*L_Y\to \rho^*L_\cX\to \rho^*L_{\cX/Y}.
\]
By construction of the cotangent complex of a stack, we have a diagram all of whose columns and rows are exact triangles:
\[
\xymatrix{
\rho^*L\pi^*L_Y\ar[r]\ar[d] & \rho^*L\pi^*L_Y\ar[r]\ar[d] & 0\ar[d]\\
\rho^*L_\cX\ar[r]\ar[d] & \Omega^1_{\widetilde{X}}\ar[r]^-{d_1}\ar[d] & \Omega^1_{\widetilde{X}/\cX}\ar[d]\\
\rho^*L_{\cX/Y}\ar[r] & L_{\widetilde{X}/Y}\ar[r] & \Omega^1_{\widetilde{X}/\cX}
}
\]
Since $\Omega^1_{\widetilde{X}/\cX}$ is a vector bundle, $L\widetilde{\varphi}^*\Omega^1_{\widetilde{X}/\cX}=\widetilde{\varphi}^*\Omega^1_{\widetilde{X}/\cX}$; similarly for $\Omega^1_{\widetilde{X}}$. We next apply $L\widetilde{\varphi}^*$ to the above diagram and take cohomology; using the fact that for $Z\to Y$ a map of schemes, $L^0\widetilde{\varphi}^*L_{Z/Y}=\widetilde{\varphi}^*\Omega^1_{Z/Y}$, we obtain a diagram all of whose rows and columns are exact sequences:
\[
\xymatrix{
 & 0\ar[d] & 0\ar[d] &  & & \\
0\ar[r] & L^{-1}\varphi^*L_{\cX/Y}\ar[r]^-{\simeq}\ar[d] & L^{-1}\varphi^*L_{\widetilde{X}/Y}\ar[d]\ar[r] & 0\ar[d] & & \\
0\ar[r] & \psi^*\Omega^1_Y\ar[r]^-{\simeq}\ar[d] & \psi^*\Omega^1_Y\ar[r]\ar[d]^-{\widetilde{\varphi}^*(d_0)} & 0\ar[d] & & \\
0 \ar[r] & L^0\varphi^*L_\cX\ar[r]\ar[d] & \widetilde{\varphi}^*\Omega^1_{\widetilde{X}}\ar[r]^-{\widetilde{\varphi}^*(d_1)}\ar[d] & \widetilde{\varphi}^*\Omega^1_{\widetilde{X}/\cX}\ar[d]^-{\simeq}\ar[r] & L^1\varphi^*L_\cX\ar[r]
\ar[d]^-{\simeq} & 0\\
0 \ar[r] & L^0\varphi^*L_{\cX/Y}\ar[r]\ar[d] & \widetilde{\varphi}^*\Omega^1_{\widetilde{X}/Y}\ar[r]\ar[d] & \widetilde{\varphi}^*\Omega^1_{\widetilde{X}/\cX}\ar[r]\ar[d] & L^1\varphi^*L_{\cX/Y}\ar[r] & 0\\
 & 0 & 0  & 0 & & \\
}
\]
From this diagram, we have natural isomorphisms
\[
\coker(\widetilde{\varphi}^*(d_1))\simeq L^1\varphi^*L_\cX\simeq L^1\varphi^*L_{\cX/Y}
\]
and 
\[
\ker(\widetilde{\varphi}^*(d_1))/\im(\widetilde{\varphi}^*(d_0))\simeq (L^0\varphi^*L_\cX)/\im(\psi^*\Omega^1_Y)\simeq L^0\varphi^*L_{\cX/Y}
\]
proving the result.
\end{proof}

We now turn to Proposition \ref{prop:ordJi-computes-hts}.

\begin{proof}[{Proof of Proposition \ref{prop:ordJi-computes-hts}}]
Let $R=k'[[t]]$. 
By Proposition \ref{prop:E-LXY-same-ht},
\[
\het^{(1)}_{L_{\cX/Y}}(\varphi)=\nhet^{(1)}_{\cE_{\widetilde{X}/\cX/Y}}(\widetilde{\varphi})=\dim_{k'}(\coker(\widetilde{\varphi}^*(d_1)))
\]
and
\[
\het^{(0)}_{L_{\cX/Y}}(\varphi)=\nhet^{(0)}_{\cE_{\widetilde{X}/\cX/Y}}(\widetilde{\varphi})=\dim_{k'}\ker(\widetilde{\varphi}^*(d_1))/\im(\widetilde{\varphi}^*(d_0)).
\]
Noting that $\widetilde{\varphi}^*\Omega^1_{\widetilde{X}/\cX}$ is free of rank $r$, $\widetilde{\varphi}^*\Omega^1_{\widetilde{X}}$ is free of rank $d+r$, and $\psi^*\Omega^1_Y$ is a finitely generated $R$-module of rank $d$, a direct application of Lemma \ref{l:Fitting-torsion-coh-length} shows
\[
\het^{(0)}_{L_{\cX/Y}}(\varphi)=\ord_t \Fitt_r(\coker(\widetilde{\varphi}^*(d_0)))
\]
and 
\[
\het^{(1)}_{L_{\cX/Y}}(\varphi)=\ord_t \Fitt_0(\coker(\widetilde{\varphi}^*(d_1))).
\]
Lastly, we simply recall that Fitting ideals pull back appropriately:
\[
\Fitt_j(\coker(\widetilde{\varphi}^*(d_i)))=\Fitt_j(\widetilde{\varphi}^*\coker(d_i))=\widetilde{\varphi}^{-1}\Fitt_j(\coker(d_i)),
\]
which finishes the proof.
\end{proof}

\section{The order function of $mK_{\cX/Y}$:~proof of Theorems \ref{thm:main-mcvf} and \ref{thm:main-canonical-divisor}}
\label{sec:relation-fitting}

We keep the notation of Section \ref{sec:hts-fitting}. Our goal is to relate $\omega_{\cX/Y,m}$, $\cJ_0$, $\cJ_1$, and $\cJ:=\sJ_{Y,m}$.


First note that $p$ is smooth and hence, we have an exact sequence
\[
0\to p^*\Omega^1_{\widetilde{U}}\to \Omega^1_{\widetilde{U}}\to \Omega^1_{\widetilde{U}/\cU}\to 0.
\]
This induces an isomorphism
\begin{equation}\label{eqn:det-map-c}
c\colon p^*\Omega^d_\cU\otimes \Omega^r_{\widetilde{U}/\cU}\to \Omega^{d+r}_{\widetilde{U}}.
\end{equation}
From Diagram \eqref{eqn:main-diag}, we then have a natural isomorphism
\[
\widetilde{\jmath}^*(\rho^*\pi^*\omega_{Y,m}\otimes (\Omega^r_{\widetilde{X}/\cX})^{\otimes m})\xrightarrow{\simeq} (p^*\Omega^d_\cU\otimes \Omega^r_{\widetilde{U}/\cU})^{\otimes m}\xrightarrow{\simeq} (\Omega^{d+r}_{\widetilde{U}})^{\otimes m},
\]
By adjunction, we then have an induced map 
\[
\iota_1\colon\rho^*\pi^*\omega_{Y,m}\otimes (\Omega^r_{\widetilde{X}/\cX})^{\otimes m}\to \widetilde{\jmath}_*(\Omega^{d+r}_{\widetilde{U}})^{\otimes m}.
\]
By adjunction, we also have a map
\[
\iota_2\colon(\Omega^{d+r}_{\widetilde{X}})^{\otimes m}\to \widetilde{\jmath}_*(\Omega^{d+r}_{\widetilde{U}})^{\otimes m}.
\]

\begin{lemma}\label{l:iotas-inj}
The maps $\iota_1$ and $\iota_2$ are injections.
\end{lemma}
\begin{proof}
The maps $\iota_i$ are generically isomorphisms, so the kernels are torsion. However, the sources of the maps are line bundles, hence torsion free. Thus, the kernels are zero.
\end{proof}

We next define two maps $\alpha_1$ and $\alpha_2$. Comparing their images will yield our desired relation among $\cJ_0$, $\cJ_1$, $\cJ$, and $\omega_{\cX/Y,m}$. First, we have a natural map
\begin{align*}
\alpha_1\colon(\rho^*\pi^*\Omega^d_Y\otimes \Omega^r_{\widetilde{X}})^{\otimes m}&\longrightarrow 
\rho^*\pi^*\omega_{Y,m}\otimes (\Omega^r_{\widetilde{X}})^{\otimes m}\\
&\xrightarrow{\id\otimes (\wedge^rd_1)^{\otimes m}}
\rho^*\pi^*\omega_{Y,m}\otimes (\Omega^r_{\widetilde{X}/\cX})^{\otimes m}\xhookrightarrow{\iota_1} \widetilde{\jmath}_*(\Omega^{d+r}_{\widetilde{U}})^{\otimes m}.
\end{align*}
On the other hand, since we have a canonical isomorphism
\[
\beta\colon(\Omega^d_{\widetilde{X}})^\vee\otimes\Omega^{d+r}_{\widetilde{X}}\xrightarrow{\simeq}\Omega^r_{\widetilde{X}}
\]
we obtain another map
\begin{align*}
\alpha_2\colon(\rho^*\pi^*\Omega^d_Y\otimes \Omega^r_{\widetilde{X}})^{\otimes m}&\xrightarrow{(\id\otimes\beta^{-1})^{\otimes m}}(\rho^*\pi^*\Omega^d_Y\otimes(\Omega^d_{\widetilde{X}})^\vee\otimes\Omega^{d+r}_{\widetilde{X}})^{\otimes m}\\
&\xrightarrow{(\partial_0\otimes\id)^{\otimes m}} (\Omega^{d+r}_{\widetilde{X}})^{\otimes m}\xhookrightarrow{\iota_2} \widetilde{\jmath}_*(\Omega^{d+r}_{\widetilde{U}})^{\otimes m},
\end{align*}
where $(\id\otimes\beta^{-1})^{\otimes m}$ is an isomorphism and $\partial_0\colon\rho^*\pi^*\Omega^d_Y\otimes(\Omega^d_{\widetilde{X}})^\vee\to\cO_{\widetilde{X}}$ is induced from $d_0$ as in Remark \ref{rmk:Fitt-E}.

\begin{lemma}\label{l:equality-of-alphas}
The images of $\alpha_1$ and $\alpha_2$ are equal.
\end{lemma}
\begin{proof}
Since $\widetilde{\jmath}_*(\Omega^{d+r}_{\widetilde{U}})^{\otimes m}$ is generically a line bundle, it is torsion-free. As a result, each map $\alpha_i$ factors as
\[
(\rho^*\pi^*\Omega^d_Y\otimes \Omega^r_{\widetilde{X}})^{\otimes m}\twoheadrightarrow \cF\xrightarrow{\overline{\alpha}_i} \widetilde{\jmath}_*(\Omega^{d+r}_{\widetilde{U}})^{\otimes m}
\]
where $\cF$ is the quotient of $(\rho^*\pi^*\Omega^d_Y\otimes \Omega^r_{\widetilde{X}})^{\otimes m}$ by its torsion subsheaf. It therefore suffices to prove that $\overline{\alpha}_1$ and $\overline{\alpha}_2$ have the same image.

Since $\widetilde{\jmath}_*(\Omega^{d+r}_{\widetilde{U}})^{\otimes m}$ is a line bundle over $\widetilde{U}$, the torsion subsheaf of $(\rho^*\pi^*\Omega^d_Y\otimes \Omega^r_{\widetilde{X}})^{\otimes m}$ is supported on $\widetilde{X}\setminus\widetilde{U}$. In particular, $\overline{\alpha}_i$ factors as
\[
\cF\hookrightarrow\widetilde{\jmath}_*\cF\xrightarrow{\widetilde{\jmath}_*\overline{\alpha}_i}\widetilde{\jmath}_*(\Omega^{d+r}_{\widetilde{U}})^{\otimes m}
\]
and so it suffices to prove $\widetilde{\jmath}^*\overline{\alpha}_1$ and $\widetilde{\jmath}^*\overline{\alpha}_2$ have the same images. Note that $\widetilde{\jmath}^*\overline{\alpha}_i=\widetilde{\jmath}^*\alpha_i$. We have therefore reduced to proving commutativity of the diagram
\[
\xymatrix{
p^*\Omega^d_{\cU}\otimes \Omega^r_{\widetilde{U}}\ar[r]^-{\id\otimes\wedge^r(d_1)} & p^*\Omega^d_{\cU}\otimes \Omega^r_{\widetilde{U}/\cU}\ar[d]^-{\simeq}_-{c}\\
p^*\Omega^d_{\cU}\otimes (\Omega^d_{\widetilde{U}})^\vee\otimes \Omega^{d+r}_{\widetilde{U}}\ar[u]^-{\simeq}_-{\id\otimes\beta}\ar[r]^-{\partial_0\otimes\id} & \Omega^{d+r}_{\widetilde{U}}
}
\]

To do so, we may work locally where all vector bundles are free. The map $c\colon p^*\Omega^d_{\cU}\otimes \Omega^r_{\widetilde{U}/\cU}\xrightarrow{\simeq}\Omega^{d+r}_{\widetilde{U}}$ is given by
\[
c(\epsilon\otimes\omega)=(\wedge^d(d_0))(\epsilon)\wedge\widetilde{\omega},
\]
where $\widetilde{\omega}\in\Omega^r_{\widetilde{U}}$ is any form such that $(\wedge^r(d_1))(\widetilde{\omega})=\omega$. In particular,
\[
c\circ(\id\otimes\wedge^r(d_1))\colon\epsilon\otimes\omega\mapsto (\wedge^d(d_0))(\epsilon)\wedge\omega.
\]
Since
\[
\partial_0(\epsilon\otimes\eta)=\eta\left((\wedge^d(d_0))(\epsilon)\right)
\]
for any $\epsilon\in p^*\Omega^d_{\cU}$ and $\eta\in(\Omega^d_{\widetilde{U}})^\vee$, commutativity of the diagram amounts to showing
\[
\eta\left((\wedge^d(d_0))(\epsilon)\right)\cdot\omega=(\wedge^d(d_0))(\epsilon)\wedge\eta(\omega)
\]
for $\omega\in\Omega^{d+r}_{\widetilde{U}}$.

It is in fact true that
\[
\eta(\gamma)\cdot\omega=\gamma\wedge\eta(\omega)
\]
for any $\gamma\in\Omega^d_{\widetilde{U}}$, $\eta\in(\Omega^d_{\widetilde{U}})^\vee$, and $\omega\in\Omega^{d+r}_{\widetilde{U}}$. Indeed, upon locally fixing a choice of basis for $\Omega^1_{\widetilde{U}}$, by linearity, we can assume $\eta$, $\gamma$, and $\omega$ are exterior products of basis vectors. One immediately checks that if $\eta$ is not dual to $\gamma$, then $\gamma\wedge\eta(\omega)=0=\eta(\gamma)\omega$. If $\eta$ is dual to $\gamma$, then let $\gamma'$ be an exterior product of basis vectors with $\omega=\gamma\wedge\gamma'$; then we see $\gamma\wedge\eta(\omega)=\gamma\wedge\gamma'=\omega=\eta(\gamma)\omega$.
\end{proof}

Next, since $\widetilde{X}$ is smooth, each connected component $\widetilde{X}_i$ is smooth and integral. Let $\xi_i$ denote the generic point of $\widetilde{X}_i$ and let $K(\widetilde{X}_i)$ denote its function field. Consider the sheaf of rational functions $\cK_{\widetilde{X}}$ which, when restricted to $\widetilde{X}_i$, is the constant sheaf associated to $K(\widetilde{X}_i)$. Since $\widetilde{\jmath}_*(\Omega^{d+r}_{\widetilde{U}})^{\otimes m}$ is torsion-free and generically a line bundle, we have canonical injections $\widetilde{\jmath}_*(\Omega^{d+r}_{\widetilde{U}})^{\otimes m}|_{\widetilde{X}_i}\to(\Omega^{d+r}_{\widetilde{U}})^{\otimes m}_{\xi_i}$. Thus, upon choosing, for each $i$, a non-zero element $\omega_i\in\Omega^{d+r}_{\widetilde{U},\xi_i}$, we have isomorphisms $(\Omega^{d+r}_{\widetilde{U}})^{\otimes m}_{\xi_i}\simeq K(\widetilde{X}_i)$, and hence an embedding
\[
\iota\colon\widetilde{\jmath}_*(\Omega^{d+r}_{\widetilde{U}})^{\otimes m}\hookrightarrow\cK_{\widetilde{X}}.
\]
The embeddings $\iota\circ\iota_i$ then endow $\rho^*\pi^*\omega_{Y,m}\otimes (\Omega^r_{\widetilde{X}/\cX})^{\otimes m}$ and $(\Omega^{d+r}_{\widetilde{X}})^{\otimes m}$ with the structure of invertible fractional ideals, and hence
\begin{equation}\label{eqn:omegaXYm}
\rho^*\omega_{\cX/Y,m}\simeq\rho^*\pi^*\omega_{Y,m}^\vee \otimes((\Omega^r_{\widetilde{X}/\cX})^\vee\otimes\Omega^{d+r}_{\widetilde{X}})^{\otimes m}
\end{equation}
with the structure of a fraction ideal as well; the isomorphism above is canonical and comes from the canonical isomorphisms
\[
\rho^*\det(L_\cX)\simeq \det(\rho^*L_{\cX})\simeq\Omega^{d+r}_{\widetilde{X}}\otimes(\Omega^r_{\widetilde{X}/\cX})^\vee.
\]
induced by the exact triangle \eqref{eq:rhoL_X}. We denote by $mK_{\widetilde{X}/\cX/Y,\iota}$ the induced Cartier divisor on $\widetilde{X}$. We now show that $mK_{\widetilde{X}/\cX/Y,\iota}$  is independent of $\iota$ and agrees with the pullback of $mK_{\cX/Y}$.

\begin{lemma}\label{l:two-defs-ofKXY-agree}
We have an equality of Cartier divisors $\rho^*(mK_{\cX/Y})=mK_{\widetilde{X}/\cX/Y,\iota}$.
\end{lemma}
\begin{proof}
We recall the definition of $mK_{\cX/Y}$ and compute its pullback under $\rho$. We have a canonical isomorphism $j^*\omega_{\cX/Y,m}^\vee\simeq\cO_{\cU}$, hence a canonical identification of the generic fiber of $\omega_{\cX/Y,m}^\vee$ with $k(\cU)$; the resulting fractional ideal structure on $\omega_{\cX/Y,m}^\vee$ 
yields the $\QQ$-divisor $-mK_{\cX/Y}$. Next, pulling back the adjunction map $\omega_{\cX/Y,m}^\vee\hookrightarrow j_*j^*\omega_{\cX/Y,m}^\vee$ yields an injection
\[
\rho^*\omega_{\cX/Y,m}^\vee\hookrightarrow \rho^*j_*j^*\omega_{\cX/Y,m}^\vee \xrightarrow{\simeq} \widetilde{\jmath}_*p^*j^*\omega_{\cX/Y,m}^\vee=\widetilde{\jmath}_*\widetilde{\jmath}^*\rho^*\omega_{\cX/Y,m}^\vee
\]
where the second map is an isomorphism by flat base change. By equations \eqref{eqn:det-map-c} and \eqref{eqn:omegaXYm}, we have canonical isomorphisms
\[
\widetilde{\jmath}^*\rho^*\omega_{\cX/Y,m}^\vee\xrightarrow{\simeq}(p^*\Omega^d_\cU\otimes (\Omega^{d+r}_{\widetilde{U}})^{\vee}\otimes \Omega^r_{\widetilde{U}/\cU})^{\otimes m}\xrightarrow{\simeq} \cO_{\widetilde{U}}.
\]
From the canonical isomorphism $\cO_{\widetilde{U},\xi_i}\simeq K(\widetilde{X}_i)$, we obtain an injection
\[
\delta\colon\rho^*\pi^*\omega_{Y,m} \otimes(\Omega^r_{\widetilde{X}/\cX}\otimes(\Omega^{d+r}_{\widetilde{X}})^\vee)^{\otimes m}\simeq\rho^*\omega_{\cX/Y,m}^\vee\hookrightarrow\cK_{\widetilde{X}}
\]
which defines the Cartier divisor $\rho^*(-mK_{\cX/Y})$. 

On the other hand, we have an embedding
\[
\delta_\iota\colon\rho^*\pi^*\omega_{\cX/Y,m}\hookrightarrow \cK_{\widetilde{X}}
\]
defining $-mK_{\widetilde{X}/\cX/Y,\iota}$. In order to prove $\delta_i=\delta$, and hence $\rho^*(mK_{\cX/Y})=mK_{\widetilde{X}/\cX/Y,\iota}$, it suffices to look locally, where we may assume $\rho^*\pi^*\omega_{Y,m}$, $\Omega^r_{\widetilde{X}/\cX}$, and $\Omega^{d+r}_{\widetilde{X}}$ are free and generated, respectively, by $\epsilon$, $\gamma$, and $\tau$. Let $\tau^\vee\in(\Omega^{d+r}_{\widetilde{X}})^\vee$ be the dual basis element to $\tau$. Then we need only check that $\delta$ and $\delta_\iota$ have the same images on the basis element $\epsilon\otimes\gamma^{\otimes m}\otimes(\tau^\vee)^{\otimes m}$. Recall that $\iota$ is defined by fixing, for each $i$, a non-zero element $\omega_i\in\Omega^{d+r}_{\widetilde{U},\xi_i}$. For each $i$, we additionally fix non-zero elements $\epsilon_i\in(p^*\Omega^d_{\cU})_{\xi_i}$ and $\widetilde{\gamma}_i\in\Omega^{d+r}_{\widetilde{U}/\cU,\xi_i}$ such that
\[
(\wedge^r(d_1))(\widetilde{\gamma}_i)=\gamma_{\xi_i}
\]
where the subscript $\xi_i$ denotes the image in the generic fiber. Let
\[
\sigma_i=(\wedge^d(d_0))(\epsilon_i)\wedge\widetilde{\gamma}_i\quad\textrm{and}\quad\epsilon_{\xi_i}=f_i\epsilon_i^{\otimes m}
\]
with $f_i\in K(\widetilde{X}_i)$. Since the isomorphism
\[
(p^*\Omega^d_\cU)_{\xi_i}\otimes \Omega^r_{\widetilde{U}/\cU,\xi_i}\otimes (\Omega^{d+r}_{\widetilde{U},\xi_i})^{\vee}\xrightarrow{\simeq} \cO_{\widetilde{U},\xi_i}
\]
maps $\epsilon_i\otimes\gamma_{\xi_i}\otimes\tau^\vee_{\xi_i}$ to $\tau_{\xi_i}^\vee(\sigma_i)$, we see
\[
\delta(\epsilon\otimes\gamma^{\otimes m}\otimes(\tau^\vee)^{\otimes m})=( f_i (\tau^\vee_{\xi_i}(\sigma_i))^m )_i\in\prod_i K(\widetilde{X}_i).
\]
Letting
\[
\sigma_i=g_i\omega_i\quad\textrm{and}\quad \tau_{\xi_i}=h_i\omega_i
\]
with $g_i,h_i\in K(\widetilde{X}_i)$, we have
\[
\delta_\iota(\epsilon\otimes\gamma^{\otimes m}\otimes(\tau^\vee)^{\otimes m})=( f_ig_i^mh_i^{-m} )_i\in\prod_i K(\widetilde{X}_i).
\]
Since $\tau^\vee_{\xi_i}=h_i^{-1}\omega_i^\vee$, we see $\tau^\vee_{\xi_i}(\sigma_i)=g_ih_i^{-1}$, thereby proving $\delta$ and $\delta_\iota$ have the same images.
\end{proof}

We may now obtain the desired relation among our ideals.

\begin{theorem}\label{thm:Fitting-ideals-key-relation}
We have an equality
\begin{equation}\label{eqn:equalityIdeals}
\cJ_0^m(\Omega^{d+r}_{\widetilde{X}})^{\otimes m}=\cJ\cJ_1^m(\rho^*\pi^*\omega_{Y,m} \otimes(\Omega^r_{\widetilde{X}/\cX})^{\otimes m})
\end{equation}
as subsheaves of $\widetilde{\jmath}_*(\Omega^{d+r}_{\widetilde{U}})^{\otimes m}$. Furthermore,
\begin{equation}\label{eqn:equalityOrderFunctions}
m\ord_{\cJ_0} - m\ord_{\cJ_1} - \ord_\cJ\circ\sL(\pi\circ\rho) = \ord_{mK_{\cX/Y}}\circ\sL(\rho).
\end{equation}
\end{theorem}
\begin{proof}
We compute the images of $\alpha_1$ and $\alpha_2$. By Remark \ref{rmk:Fitt-E}, the image of $\partial_0$ is $\cJ_0$, and so $\alpha_2$ factors as
\begin{align*}
(\rho^*\pi^*\Omega^d_Y\otimes\Omega^r_{\widetilde{X}})^{\otimes m} &\xrightarrow{\simeq}(\rho^*\pi^*\Omega^d_Y\otimes(\Omega^d_{\widetilde{X}})^\vee\otimes\Omega^{d+r}_{\widetilde{X}})^{\otimes m}\\
&\twoheadrightarrow (\cJ_0\otimes\Omega^{d+r}_{\widetilde{X}})^{\otimes m}
\twoheadrightarrow \cJ_0^m(\Omega^{d+r}_{\widetilde{X}})^{\otimes m}
\xhookrightarrow{\iota_2} \widetilde{\jmath}_*(\Omega^{d+r}_{\widetilde{U}})^{\otimes m}.
\end{align*}
On the other hand, using the definition of $\cJ$ and Remark \ref{rmk:Fitt-E}, we see $\alpha_1$ factors as 
\begin{align*}
(\rho^*\pi^*\Omega^d_Y)^{\otimes m}\otimes\Omega^r_{\widetilde{X}}&\twoheadrightarrow \cJ\cO_{\widetilde{X}}\rho^*\pi^*\omega_{Y,m} \otimes(\cJ_1\Omega^r_{\widetilde{X}/\cX})^{\otimes m}\\
&\twoheadrightarrow \cJ\cJ_1^m(\rho^*\pi^*\omega_{Y,m} \otimes(\Omega^r_{\widetilde{X}/\cX})^{\otimes m})
\xhookrightarrow{\iota_1} 
\widetilde{\jmath}_*(\Omega^{d+r}_{\widetilde{U}})^{\otimes m}.
\end{align*}
By Lemma \ref{l:equality-of-alphas}, we have our desired equality \eqref{eqn:equalityIdeals} as subsheaves of $\widetilde{\jmath}_*(\Omega^{d+r}_{\widetilde{U}})^{\otimes m}$.

Considering the order function associated to $mK_{\widetilde{X}/\cX/Y,\iota}$, equality \eqref{eqn:equalityIdeals} immediately implies that the lefthand side of \eqref{eqn:equalityOrderFunctions} is equal to $\ord_{mK_{\widetilde{X}/\cX/Y,\iota}}$. This latter quantity is then equal to $\ord_{mK_{\cX/Y}}\circ\sL(\rho)$ by Lemma \ref{l:two-defs-ofKXY-agree}.
\end{proof}

We are now ready to prove two of the main theorems of this paper.

\begin{proof}[{Proof of Theorem \ref{thm:main-canonical-divisor}}]

In light of Proposition \ref{prop:ordJi-computes-hts}, we may rewrite equation \eqref{eqn:equalityOrderFunctions} as
\begin{equation}\label{eqn:almost-main-them}
m\het^{(0)}_{L_{\cX/Y}}\circ\sL(\rho) -m\het^{(1)}_{L_{\cX/Y}}\circ\sL(\rho) -\ord_\cJ\circ\sL(\pi\circ\rho)=\ord_{mK_{\cX/Y}}\circ\sL(\rho).
\end{equation}
%
Given any arc $\varphi\colon \Spec k'[[t]]\to\cX$, the two sides of 
this equation are not altered when replacing $k'$ by a larger field extension. After such an extension, we may assume the closed point of $\varphi$ lifts to $\widetilde{X}$, and hence by smoothness of $\rho$, we obtain a lift $\widetilde\varphi\colon D\to\widetilde{X}$ of $\varphi$. Evaluating \eqref{eqn:almost-main-them} 
on $\widetilde\varphi$ establishes the result.
\end{proof}

\begin{proof}[{Proof of Theorem \ref{thm:main-mcvf}}]
This is immediate from Theorems \ref{thm:main-canonical-divisor} and \ref{theoremChangeOfVariables-reformulation}.
\end{proof}

\section{Order functions determine Cartier divisors}\label{sec:orderfnc->cartier}

In this section, we prove that a Cartier divisor is determined by its order function away from a closed substack.

\begin{proposition}\label{prop:order-fnc-determines-div}
Let $\cY$ be a smooth irreducible finite type Artin stack over $k$, let $\cD_1$ and $\cD_2$ be Cartier divisors on $\cY$, and let $\cZ\subset\cY$ be a closed substack such that $|\cZ|\subsetneq|\cY|$. Set $\cZ'$ to be the union of $\cZ$ and the supports of $\cD_1$ and $\cD_2$. If $\ord_{\cD_1} = \ord_{\cD_2}$ on $|\sL(\cY) \setminus \sL(\cZ')|$, then $\cD_1 = \cD_2$.
\end{proposition}
\begin{proof}
Let $\rho\colon Y\to\cY$ be a smoth cover by a scheme. Then each connected component of $V$ is a smooth integral finite type scheme. By considering $Z:=\cZ\times_Y\cY$ and $D_i:=\rho^*\cD_i$ restricted to each component of $Y$, we see it suffices to prove the result when $\cY=Y$. 

The statement only depends on the support of $Z'$, so we replace $Z'$ with its reduced structure to assume its irreducible components are generically smooth. Now, it is sufficient to show for every irreducible divisor $D$ supported in $Z'$, that $D_1$ and $D_2$ have the same multiplicity at $D$. To do this, we consider an arc $\phi\in\cL(X)$ whose generic point maps to $X \setminus Z'$ and whose special point maps to $D \setminus D'$, where $D'$ denotes the union of the components of $D_1$ and $D_2$ not contained in $D$. To see that such an arc exists, consider a smooth point $x\in D \setminus D''$, where $D''$ denotes the union of the components of $Z'$, $D_1$ and $D_2$ not contained in $D$. Then we have a strict inequality of tangent space dimenisions $\dim T_{Z',x}<\dim T_{X,x}$, so there is a $1$-jet $\phi_1$ that does not factor through $Z'$ but whose special point maps to $x$. 
Since $X$ is smooth, $\phi_1$ lifts to our desired arc $\phi$. 

Lastly, since the special point of $\phi$ maps to $x$ which is in $D$ but not in any other component of $D_i$, the multiplicity of $D_i$ at $D$ is given by $\ord_{D_i}(\phi)/\ord_D(\phi)$, which is independent of $i$ since $\ord_{D_1}(\phi) = \ord_{D_2}(\phi)$ and $\phi\in |\sL(X)\setminus \sL(X \setminus Z')|$.
\end{proof}

\begin{proof}[{Proof of Corollary \ref{cor:triangle-iff-crepant}}]
Let $\cU\hookrightarrow \cX$ be the open substack in Theorem \ref{thm:main-canonical-divisor}. If $\pi$ is crepant, then \eqref{eqn:condition-triangle} holds with $\cV=\cU$. Conversely, if $\cV\hookrightarrow \cX$ is the open substack in \eqref{eqn:condition-triangle}, then $\ord_{mK_{\cX/Y}}=0$ on $|\sL(\cX)|\setminus|\sL(\cZ)|$, where $\cZ=\cX\setminus(\cU\cap\cV)$. By Proposition \ref{prop:order-fnc-determines-div}, we see $mK_{\cX/Y}=0$.
\end{proof}

\section{The case where $Y$ is lci}\label{sectionLciCase}

It is worth pointing out that in the case where $Y$ is lci, there is an alternative and aesthetically pleasing proof of Theorem \ref{thm:main-mcvf}. The key idea is that $L_{\cX/Y}$ is a perfect complex whose cohomology is concentrated in degrees $[-1,1]$ and that $\het^{(-1)}_{L_{\cX/Y}}$ agrees with the function $\ord_\cJ\circ\sL(\pi)$; as a result, the alternating sum of its height functions precisely recovers our main quantity of interest:
\[
\sum_{i\in\Z}(-1)^i\het^{(i)}_{L_{\cX/Y}} = \het^{(0)}_{L_{\cX/Y}} - \het^{(1)}_{L_{\cX/Y}} - \ord_\cJ\circ\sL(\pi).
\]

\begin{proposition}\label{prop:Ylci->perfect-ht-2-vanish}
Using Notation \ref{running-notation}, if $Y$ is lci, then $L_{\cX/Y}$ is a perfect complex. Moreover,
\[
\het^{(-1)}_{L_{\cX/Y}}=\ord_{\cJ_Y}\circ\sL(\pi)
\]
and $\het^{(i)}_{L_{\cX/Y}}=0$ for all $i\notin\{-1,0,1\}$. 
\end{proposition}
\begin{proof}
Fix an arc $\varphi\colon D=\Spec R\to\cX$ with $R=k'[[t]]$ and let $\psi=\pi\circ\varphi$. By \cite[Lemma 5.3]{SatrianoUsatine2}, $L^{-1}\varphi^*L_{\cX/Y}=(\psi^*\Omega^1_Y)_{\tors}$. So, $\het^{(-1)}_{L_{\cX/Y}}(\varphi)=\dim_{k'} (\psi^*\Omega^1_Y)_{\tors}$. For any finitely generated $R$-module $M$ of rank $d$, it is easy to check that $\ord_t\Fitt_d(M)=\dim_{k'}(M_{\tors})$, so we see 
\[
\het^{(-1)}_{L_{\cX/Y}}(\varphi)=\ord_{\Fitt_d(\Omega^1_Y)}(\psi)=\ord_{\cJ_Y}(\psi),
\]
where the last equality comes from \cite[Remark 9.6]{EinMustata}.

Next, we have an exact triangle
\[
L\pi^*L_Y\to L_\cX\to L_{\cX/Y}.
\]
Since $Y$ is an lci, $L_Y$ is perfect and supported in degrees $[-1,0]$. Since $\cX$ is smooth, $L_\cX$ is perfect, supported in degrees $[0,1]$. It follows that $L_{\cX/Y}$ is perfect supported in degrees $[-2,1]$.

To finish the proof, it remains to prove $\het^{(-2)}_{L_{\cX/Y}}=0$. 
By \cite[Lemma 5.3]{SatrianoUsatine2}, we know $L^{-2}\varphi^*L_{\cX/Y}$ is torsion, so it is enough to prove that it is also free. For this, we observe that by the exact triangle above, $L^{-2}\varphi^*L_{\cX/Y}=L^{-1}\psi^*L_Y$. Now since $Y$ is lci, after possibly shrinking $Y$, we may assume there is a regular closed immersion $i\colon Y\to Z$ defined by an ideal sheaf $\cI$ where $Z$ is smooth over $k$. Then there is an exact triangle
\[
\cI/\cI^2\to i^*\Omega^1_Z\to L_Y
\]
from which we have an exact sequence
\[
0\to L^{-1}\psi^*L_Y\to \psi^*(\cI/\cI^2)\to \psi^*i^*\Omega^1_Z.
\]
Since $\cI/\cI^2$ is locally free, $\psi^*(\cI/\cI^2)$ is a free $R$-module. It follows that $L^{-1}\psi^*L_Y$ is torsion-free, and hence also a free $R$-module.
\end{proof}

Next, we prove the following general result which relates the height functions of a generically exact perfect complex to the order function of its associated Cartier divisor.

\begin{proposition}\label{prop:alternating-sum-hts-ord}
Let $\cY$ be an irreducible Artin stack over $k$ and let $\jmath\colon\cV\hookrightarrow\cY$ be an open dense subscheme. Suppose $\cE\in D^b_{\coh}(\cY)$ is perfect and $\jmath^*\cE\simeq0$. Then we obtain canonical injections $\iota\colon\det(\cE)\hookrightarrow \jmath_*\det(\jmath^*\cE)\xrightarrow{\simeq} \jmath_*\cO_{\cV}$, and hence an associated Cartier divisor that we denote by $\Div(\cE)$, see Proposition \ref{prop:rat-sec-->Cartier}.

Suppose $\varphi\colon\Spec k'[[t]]\to\cY$ is an arc, where $k'/k$ is a field extension, and that the generic point of $\varphi$ maps to $\cV$. Then
\[
\ord_{\Div(\cE)}(\varphi)=\sum_{i\in\bZ}(-1)^i\het^{(i)}_\cE(\varphi).
\]
\end{proposition}
\begin{proof}
Since $\det(\cE)$ is a line bundle, we have a canonical injection $\det(\cE)\hookrightarrow \jmath_*\jmath^*\det(\cE)\xrightarrow{\simeq}\jmath_*\det(\jmath^*\cE)$. Since $\jmath^*\cE\simeq0$, the map from $\cE$ to the zero complex yields a canonical isomorphism $\det(\jmath^*\cE)\xrightarrow{\simeq} \det(0)=\cO_{\cV}$.

Let $R=k'[[t]]$ and $K=k'((t))$. To compute $\ord_{\Div(\cE)}(\varphi)$, we must choose an $R$-basis for $\varphi^*\det(\cE)=\det(L\varphi^*\cE)$ and compute the image of this basis under the map $\varphi^*(\iota)\colon\varphi^*\det(\cE)=L\varphi^*\det(\cE)\to L\varphi^*\jmath_*\cO_\cV\xrightarrow{\simeq} K$. The object $L\varphi^*\cE$ is represented by a finite complex $M^\bullet$ of free $R$-modules whose cohomology groups are torsion. Choosing a basis for each $M^i$, we obtain a based complex in the terminology of \cite[Appendix A]{GelfandKapranovZelevinsky}. Since $H^i(M^\bullet)=L^i\varphi^*\cE$, by Theorem 30 of (loc.~cit.),
\[
\ord_{\Div(\cE)}(\varphi)=\sum_{i\in\bZ}(-1)^i\dim_{k'}H^i(M^\bullet)=\sum_{i\in\bZ}(-1)^i\het^{(i)}_\cE(\varphi).\qedhere
\]
\end{proof}

We now turn to Theorem \ref{thm:main-lci-case}.

\begin{proof}[{Proof of Theorem \ref{thm:main-lci-case}}]
Propositions \ref{prop:Ylci->perfect-ht-2-vanish} and \ref{prop:alternating-sum-hts-ord} tell us
\begin{equation}\label{eqn:lci-alternating-sum-ord}
\ord_{\Div(L_{\cX/Y})}=\sum_{i\in\Z} (-1)^i\het^{(i)}_{L_{\cX/Y}}=\het^{(0)}_{L_{\cX/Y}} - \het^{(1)}_{L_{\cX/Y}} - \ord_{\sJ_Y}.
\end{equation}
This order function agrees with that of $K_{\cX/Y}$ by Theorem \ref{thm:main-canonical-divisor}; note that since $Y$ is lci, it is Gorenstein and hence $m=1$. It follows from Proposition \ref{prop:order-fnc-determines-div} that $\Div(L_{\cX/Y})=K_{\cX/Y}$.
\end{proof}

\begin{remark}
Section \ref{sectionLciCase} provides a short and completely self-contained proof of the motivic change of variables formula in the lci case if one replaces $mK_{\cX/Y}$ by $\Div(L_{\cX/Y})$ in the statement of Theorem \ref{thm:main-mcvf}. This is given by combining Theorem \ref{theoremChangeOfVariables-reformulation} with the proof of \eqref{eqn:lci-alternating-sum-ord}.
\end{remark}

\section{Examples}
\label{sec:examples}

In the following example, we construct a quotient variety that admits a crepant resolution by a smooth Artin stack, but \emph{does not} admit a crepant resolution by a smooth Deligne--Mumford stack and does not admit any non-commutative crepant resolution (NCCR). 
Thus, our motivic change of variables formula (Theorem \ref{thm:main-mcvf}) applies to a broader class of varieties than has been considered by previous techniques.

\begin{example}[{No NCCR or DM crepant resolution}]
\label{ex:noDM-or-NCCR-res}
Let $Y$ be the affine cone over the Grassmannian $\Gr(2,r)$ with respect to the Pl\"ucker embedding. The variety $Y$ is 
log-terminal, see e.g, \cite[Example 5.1]{Batyrev1998}.

Let $\cX=[\bA^{2r}/\SL_2]$ where $\bA^{2r}$ is viewed as the space of $2\times r$ matrices and $\SL_2$ acts by left multiplication. Then, by \cite[Remark 10.1]{SatrianoUsatine} and the preceding discussion, there is a good moduli space $\pi\colon\cX\to Y$, and $\pi$ is a small resolution for $r\geq3$.

On the other hand, $Y$ admits no crepant resolution by a smooth Deligne--Mumford stack by the second paragraph of \cite[Remark 5.9]{vanDenBergh22} and admits no NCCR by \cite[Example 10.2]{SpenkoVanDenBergh17}.
\end{example}

The next example illustrates that, unlike the case of schemes, when $Y$ is smooth $K_{\cX/Y}$ need not be effective.

\begin{example}[{Smooth $Y$ with anti-effective $K_{\cX/Y}$}]
\label{ex:KcXY-not-effective}
Let $Y=\bA^1$ and let $\cX=[\bA^4/\SL_2]$ where $\bA^4$ is identified with the space of $2\times2$ matrices and $\SL_2$ acts by left multiplication. Then we have a good moduli space map $\pi\colon\cX\to Y$ where a matrix $A\in\bA^4$ maps to $\det(A)\in\bA^1$. The exceptional locus $\cD$ of $\pi$ is given by $[Z/\SL_2]$ where $Z\subset\bA^4$ consists of all $2\times 2$ matrices with vanishing determinant. 
We know $K_{\cX/Y}=n\cD$ for some integer $n$.

We show $n=-1$ by applying our motivic change of variables formula Theorem \ref{thm:main-mcvf}. Following the notation of \cite[\S10]{SatrianoUsatine}, we consider the cylinders $\cC:=\cC^{(1,2)}\subset |\sL(\cX)|$ and $C:=C^{(2)}\subset \sL(Y)$. The latter cylinder consists of arcs vanishing to order at least $2$ at the origin and hence has volume
\[
\mu^{\Gor}_Y(C)=\mu_Y(C)=(\bL-1)\bL^{-3};
\]
the volume and Gorenstein volume coincide here since $Y$ is smooth. The former cylinder consists of arcs corresponding to the $\SL_2$-orbit of diagonal matrices of the form $\diag(ft,t)$ with $f\in k[[t]]^*$; its volume is computed in \cite[Proposition 10.10]{SatrianoUsatine}:
\[
\mu_{\cX}(\cC)=(\bL-1)\bL^{-5}.
\]
It is immediate from the definition that $\ord_\cD = 2$ on $\cC$ and that the generic point of every arc in $\cC$ maps to $\cX\setminus\cD$. Furthermore, the proof of \cite[Proposition 10.11]{SatrianoUsatine} shows that for all field extensions $k'/k$, the map $\overline{\cC}(k')\to C(k')$ is bijective. Thus, the hypotheses of Theorem \ref{thm:main-mcvf} are satisfied. We therefore have
\begin{align*}
(\bL-1)\bL^{-3}=\mu^{\Gor}_Y(C)&=\int_{\cC}\bL^{-\ord_{K_{\cX/Y}}}d\mu_\cX\\
&=\bL^{-2n}\mu_\cX(\cC)=(\bL-1)\bL^{-2n-5}.
\end{align*}
It follows that
\[
K_{\cX/Y}=-\cD
\]
and hence, $-K_{\cX/Y}$ is effective.

Such phenomena never occur when $\cX$ is a scheme. Indeed the correction factor in Theorem \ref{thm:main-mcvf} relating $\mu_Y=\mu^{\Gor}_Y$ and $\mu_\cX$ arises from the dimensions of fibers of the jet space maps $\sL_n(\cX)\to\sL(Y)$. For schemes these dimensions are non-negative and so $K_{\cX/Y}$ is effective; however for stacks, these fibers can be negative-dimensional. We note that an upgraded version of this example is essential to our proof of Theorem \ref{thm:crepant-log-terminal} in Section \ref{sec:crepant-log-terminal}.
\end{example}

The next example illustrates that, even when one restricts attention to varieties, it is important to define crepantness of $\pi$ using triviality of $mK_{\cX/Y}$ as opposed to triviality of $\omega_{\cX/Y}$.

\begin{example}[{Non-crepant $\pi$ with trivial $\omega_{\cX/Y}$}]
\label{ex:crepant-distinction}
Let $X'\to Y$ be any birational map of varieties where $Y$ is $\QQ$-Gorenstein and $K_{X'/Y}$ is a non-zero Cartier divisor, e.g., $X'$ could be the blow up of a smooth variety $Y$ at a smooth center. Since $K_{X'/Y}$ is locally principal, we may take $X\subset X'$ to be an open subset that intersects the support of $K_{X'/Y}$ and where $K_{X'/Y}$ is principal. By construction, $\pi\colon X\to Y$ is a weakly birational map where $\omega_{X/Y} \simeq \cO_X$ but $K_{X/Y}\neq0$.
\end{example}

\section{Crepant stacky resolutions of log-terminal singularities}
\label{sec:crepant-log-terminal}

In this section, we prove Theorem \ref{thm:crepant-log-terminal}. The proof makes crucial use of the following stack which parameterizing degenerations of framed bundles. Recall that $[\bA^1/\bG_m]$ classifies line bundles with section.

\begin{lemma}\label{l:moduli-framed-bundles}
Viewing $\bA^{r^2}$ as the variety of $r\times r$ matrices, consider the stack $\sM_r:=[\bA^{r^2}/\GL_r]$. Then the following hold.
\begin{enumerate}
\item\label{l:moduli-framed-bundles::moduli-interp} For any scheme $T$, the groupoid $\sM_r(T)$ consists of pairs $(\cE,\alpha)$ where $\cE$ is a rank $r$ vector bundle $\alpha\colon\cO^{\oplus r}\to\cE$ is a morphism. Maps of pairs are isomorphisms of vector bundles which respect the maps from $\cO^{\oplus r}$.

\item\label{l:moduli-framed-bundles::birational} Let
\[
\det\colon\sM_r\to[\bA^1/\bG_m]
\]
be the map sending $(\cE,\alpha)$ to the line bundle $\det(\cE)$ equipped with the global section $\det(\alpha)\colon\cO\to\det(\cE)$. Then $\det$ is a relative good moduli space map which is an isomorphism over the dense open point $\Spec k=[\bG_m/\bG_m]\to[\bA^1/\bG_m]$.

\item\label{l:moduli-framed-bundles::pullback} We have a cartesian diagram
\[
\xymatrix{
\sS_r\ar[r]\ar[d] & \sM_r\ar[d]^-{\det}\\
\bA^1\ar[r] & [\bA^1/\bG_m]
}
\]
where $\sS_r:=[\bA^{r^2}/\SL_r]$.
\end{enumerate}
\end{lemma}
\begin{proof}
For (\ref{l:moduli-framed-bundles::moduli-interp}), if $\cF$ denotes the universal rank $r$ vector bundle, then we must show $\sM_r$ is the total space of $\cF^{\oplus r}$. The bundle $\cF$ corresponds to the universal $\GL_r$-torsor $\Spec k\to B\GL_r$, and so $\cF^{\oplus r}$ is given by the pushout $\Spec k\times_{\GL_r}\bA^{r^2}=[\bA^{r^2}/\GL_r]=\sM_r$.

We next prove (\ref{l:moduli-framed-bundles::pullback}). For any scheme $T$, the map $\bA^1\to[\bA^1/\bG_m]$ sends $f\in\bA^1(T)$ to $(\cO_T,f)$. So, the fiber product classifies tuples $(\cE,\alpha,\iota)$ where $(\cE,\alpha)$ is as in (\ref{l:moduli-framed-bundles::moduli-interp}) and $\iota\colon\det(\cE)\xrightarrow{\simeq}\cO$ is an isomorphism. Thus, the automorphisms of $(\cE,\alpha,\iota)$ are precisely the automorphisms of $(\cE,\alpha)$ with trivial determinant.

Lastly, we turn to (\ref{l:moduli-framed-bundles::birational}). First note that $\sS_r\to\bA^1$ is a good moduli space map since the determinant is the unique invariant for $\SL_r$ acting on $r\times r$ matrices. Since $\bA^1\to[\bA^1/\bG_m]$ is a smooth cover, it follows from (\ref{l:moduli-framed-bundles::pullback}) that $\det$ is a relative good moduli space map. To finish the proof of (\ref{l:moduli-framed-bundles::birational}), fix an object $(\cE,\alpha)$ of $\sM_r(T)$, where $T$ is a scheme. We must show that if the section $\det(\alpha)\colon\cO\to\det(\cE)$ is nowhere vanishing, then $(\cE,\alpha)$ is uniquely isomorphic to $(\cO^{\oplus r},\id)$, where $\id$ is the identity map. By a standard limit argument, we may reduce to the case where $T$ is Noetherian. Since $\det(\alpha)$ is nowhere vanishing, $\alpha$ is a map of rank $r$ vector bundles which is an isomorphism on all fibers. Then by Nakayama's Lemma, $\alpha$ is an isomorphism. Thus, $\alpha\colon\cO^{\oplus r}\to\cE$ defines an isomorphism $(\cO^{\oplus r},\id)\xrightarrow{\simeq} (\cE,\alpha)$. Moreover, this is the unique such isomorphism.
\end{proof}

Our first goal is to compute $K_{\sM_r/[\bA^1/\bG_m]}$ by generalizing Example \ref{ex:KcXY-not-effective}. Note that if $\pi\colon\cX\to\cY$ is any morphism of smooth irreducible finite type Artin stacks and if $\cU\subset\cX$ is a non-empty open \emph{subscheme} such that $\cU\to\cY$ is an open immersion, then $\omega_{\cX/\cY}:=\det(L_\cX)\otimes\pi^*\det(L_{\cY})$ is canonically trivialized over $\cU$ hence defines a Cartier divisor $K_{\cX/\cY}$ by Proposition \ref{prop:rat-sec-->Cartier}.

We next understand the behaviour of the relative canonical divisor under pullbacks, products, and compositions. 

\begin{lemma}\label{l:pullback-K-lci}
The following statements hold.
\begin{enumerate}
\item\label{l:pullback-K-lci::pullback} Let
\[
\xymatrix{
\cX'\ar[r]^-{\Psi}\ar[d] & \cX\ar[d]\\
\cY'\ar[r] & \cY
}
\]
be a cartesian diagram of smooth finite type Artin stacks. Suppose $\cU\subset\cX$ is a dense open subscheme with $\cU\to\cY$ an open immersion, and suppose $\cU\times_\cX\cX'$ contains a dense open subscheme $\cU'$. Then
\[
\Psi^*K_{\cX/\cY}=K_{\cX'/\cY'}.
\]

\item\label{l:pullback-K-lci::div} Let $\cX\to\cY$ be a map of smooth finite type Artin stack and suppose $\cU\subset\cX$ is a dense open subscheme with $\cU\to\cY$ an open immerison. Then the Cartier divisor $\Div(L_{\cX/\cY})$ exists by Proposition \ref{prop:alternating-sum-hts-ord} and agrees with $K_{\cX/\cY}$.

\item\label{l:pullback-K-lci::product} For $1\leq i\leq n$, let $\cX_i\to \cX'_i$ be maps between smooth finite type Artin stacks and suppose $\cU_i\subset\cX_i$ is a dense open subscheme with $\cU_i\to\cX'_i$ an open immersion. 
Let $\cX=\prod_{i=1}^n\cX_i$, let $\cX'=\prod_{i=1}^n \cX'_i$, and let $p_i\colon\cX\to \cX_i$ denote the projection map. Then $\cU:=\prod_{i=1}^n\cU_i$ is a dense open subscheme of $\cX$ and $\cU\to \cX'$ is an open immersion. Furthermore,
\[
K_{\cX/\cX'}=\sum_{i=1}^n p_i^*K_{\cX_i/\cX'_i}.
\]

\item\label{l:pullback-K-lci::composition} Let $\cX\xrightarrow{\pi}\cY\xrightarrow{p}\cZ$ be morphisms between finite type Artin stacks with $\cX$ and $\cY$ smooth. Let $\cU\subset\cX$ and $\cV\subset\cY$ be dense open subschemes such that $\cU\to\cY$ and $\cV\to\cZ$ are open immersions. Suppose either that $\cZ$ is an $m$-Gorenstein scheme, or that $\cZ$ is smooth (in which case we take $m=1$). Then 
\[
mK_{\cX/\cZ}=mK_{\cX/\cY}+\pi^*(mK_{\cY/\cZ}).
\]
\end{enumerate}
\end{lemma}
\begin{proof}
We start by showing (\ref{l:pullback-K-lci::pullback}). Letting $\cV'=\cU\times_\cX\cX'$, we see $\cV'\subset\cX$ is an open substack with $\cV'\to\cY'$ an open immersion, and so $\cU'\to\cV'\to\cY'$ is an open immersion. We see $\Psi^*\omega_{\cX/\cY}=\omega_{\cX'/\cY'}$. The Cartier divisor $K_{\cX/\cY}$ is defined by the canonical trivialization of $\omega_{\cX/\cY}|_\cU$ which, since $\Psi(\cU')\subset\cU$, pulls back to the canonical trivialization of $\omega_{\cX'/\cY'}|_{\cU'}$. Therefore, $\Psi^*K_{\cX/\cY}=K_{\cX'/\cY'}$.

We now turn to (\ref{l:pullback-K-lci::div}). Since $\cX$ and $\cY$ are smooth, $L_{\cX/\cY}$ is perfect. Since $L_{\cX/\cY}|_\cU$ is canonically trivialized, Proposition \ref{prop:alternating-sum-hts-ord} yields a Cartier divisor $\Div(L_{\cX/\cY})$. Let $Y'=\cY'\to\cY$ be a smooth cover by a scheme, let $\cX':=\cX\times_\cY Y'$, and let $\cU'=\cU\times_\cX\cX'$. Then $\cU'\to Y'$ is an open immersion, hence a scheme. Let $\Psi\colon\cX'\to\cX$ be the induced map, which is smooth, so $\cX'$ is smooth. Then the canonical trivialization of $L_{\cX/\cY}|_\cU$ pulls back to the canonical trivialization of $L_{\cX'/Y'}|_{\cU'}$. Thus, $\Psi^*\Div(L_{\cX/\cY})=\Div(L_{\cX'/Y'})=K_{\cX'/Y'}$ where the last equality is by Theorem \ref{thm:main-lci-case}. On the other hand, applying (\ref{l:pullback-K-lci::pullback}), we find $\Psi^*K_{\cX/\cY}=K_{\cX'/Y'}=\Psi^*\Div(L_{\cX/\cY})$. Proposition \ref{prop:order-fnc-determines-div} then shows $K_{\cX/\cY}=\Div(L_{\cX/\cY})$.

To prove (\ref{l:pullback-K-lci::product}), by induction, it suffices to handle the case when $n=2$. Consider the cartesian diagram
\[
\xymatrix{
\cX\ar[r]^-{q}\ar[d]\ar@/_1.1pc/[dd]_{p_1} & \cY_2\ar[r]^-{q'}\ar[d] & \cX_2\ar[d]\\
\cY_1\ar[r]\ar[d] & \cX'\ar[r]\ar[d] & \cX'_2\ar[d]\\
\cX_1\ar[r] & \cX'_1\ar[r] & \Spec k
}
\]
Let $\cV_i:=\cU_i\times_{\cX_i}\cY_i$ and $\cW_i:=\cU_i\times_{\cX_i}\cX$. Then $\cU:=\cW_1\times_\cX\cW_2$. The map $\cU\to\cV_2$ is the pullback of $\cU_1\to \cX'_1$, hence an open immersion; the map $\cV_2\to \cX'$ is the pullback of $\cU_2\to \cX'_2$, hence also an open immersion. It follows that $\cU\to \cX'$ is an open immersion.

Next, we have an exact triangle
\[
Lq^*L_{\cY_2/\cX'}\to L_{\cX/\cX'}\to L_{\cX/\cY_2}.
\]
Using that $Lq^*L_{\cY_2/\cX'}=Lq^*L(q')^*L_{\cX_2/\cX'_2}=Lp_2^*L_{\cX_2/\cX'_2}$ and $L_{\cX/\cY_2}=Lp_1^*L_{\cX_1/\cX'_1}$, we find
\[
\det(L_{\cX/\cX'})\simeq p_1^*\det(L_{\cX_1/\cX'_1})\otimes p_2^*\det(L_{\cX_2/\cX'_2}).
\]
The Cartier divisor $\Div(L_{\cX_i/\cX'_i})$ is induced by the canonical trivialization of $L_{\cX_i/\cX'_i}|_{\cU_i}$, and $\Div(L_{\cX/\cX'})$ is induced by the canonical trivialization of $L_{\cX/\cX'}|_{\cU}$. Since $p_i(\cU)\subset\cU_i$, the above isomorphism of determinants induces an equality of Cartier divisors
\[
\Div(L_{\cX/\cX'})\simeq p_1^*\Div(L_{\cX_1/\cX'_1})\otimes p_2^*\Div(L_{\cX_2/\cX'_2})
\]
and hence, by (\ref{l:pullback-K-lci::div}), an equality
\[
K_{\cX/\cX'}=p_1^*K_{\cX_1/\cX'_1}+p_2^*K_{\cX_2/\cX'_2}.
\]

Lastly, we prove (\ref{l:pullback-K-lci::composition}). Replacing $\cU$ by $\cU\times_\cX\pi^{-1}(\cV)$, we may assume $\pi(\cU)\subset\cV$, and hence $\cU\to\cZ$ is an open immersion. Thus, $mK_{\cX/\cZ}$ is defined. We see
\begin{align*}
\omega_{\cX/\cY}^{\otimes m}&\otimes\pi^*\omega_{\cY/\cZ,m}\\
&=\det(L_\cX)^{\otimes m}\otimes\pi^*(\det(L_{\cY})^{\otimes m})^\vee \otimes \pi^*(\det(L_{\cY})^{\otimes m}\otimes p^*\omega_{\cY/\cZ,m}^\vee)=\omega_{\cX/\cZ,m}.
\end{align*}
Moreover, the open immersion $\cU\to\cV$ induces an isomorphism on generic points, so the canonical trivialization of $\omega_{\cX/\cZ,m}|_{\cU}$ is induced by the canonical trivializations of $\omega_{\cX/\cY}|_\cU$ and $\omega_{\cY/\cZ,m}|_{\cV}$. As a result, we obtain our desired equality of Cartier divisors.
\end{proof}

To compute $K_{\sM_r/[\bA^1/\bG_m]}$, we apply our motivic change of variables formula Theorem \ref{thm:main-mcvf} to specific choices of cylinders. The following technical lemma computes the appropriate measures of these cylinders.

\begin{lemma}\label{l:measures-cylinders-for-mcvf}
Let $C\subset\sL(\bA^1)$ be the cylinder of arcs vanishing to order $1$ at the origin. Let $\cC\subset|\sL(\sS_r)|$ be the set of arcs corresponding to the $\SL_r$-orbit of diagonal matrices of the form $\diag(f,1,\dots,1)$ where $f$ has valuation $1$. Then $\cC$ is a cylinder,
\[
\mu_{\sS_r}(\cC)=(\bL-1)\bL^{-(r+1)},\quad\textrm{and}\quad \mu^{\Gor}_{\bA^1}(C)=(\bL-1)\bL^{-2}.
\]
Furthermore, $\overline{\cC}(k')\to C(k')$ is bijective for all field extensions $k'/k$.
\end{lemma}
\begin{proof}
Let us first handle the bijectivity statement. If $\psi\in \cC(k')$ then up to $\SL_r$-equivalence, it is of the form $\diag(f,1,\dots,1)$ with $f\in k'[[t]]$ of valuation $1$. Then $\sL(\det)(\psi)$ is the arc defined by $f$ and hence $\psi$ is determined up to $\SL_r$-equivalence by $\sL(\det)(\psi)$.

Next, since $\bA^1$ is smooth, we have $\mu^{\Gor}_{\bA^1}(C)=\mu_{\bA^1}(C)$ and it is immediate that the latter quantity is $(\bL-1)\bL^{-2}$.

It remains to compute $\mu_{\sS_r}(\cC)$. The proof is essentially the same as \cite[Proposition 10.10]{SatrianoUsatine}, except we are taking $i=0$ and considering $r\times r$ matrices instead of $2\times r$ matrices. Let $Z\subset\sL(\bA^{r^2})$ be the set of arcs of the form $\diag(f,1\dots,1)$ with $f$ of valuation $1$ and let $Z_n\subset\sL_n(\bA^{r^2})$ be the locally closed subscheme whose $A$-valued points for any $k$-algebra $A$, are arcs of the form $\diag(ut,1\dots,1)$ where $u\in A[t]/(t^{n+1})$ is a unit. Letting $\rho\colon\bA^{r^2}\to\sS_r$ be the smooth cover, we have $\cC=\sL(\rho)(Z)$ and $\theta_n(\cC)=\sL_n(\rho)(Z_n)$ where $\theta_n$ denotes the truncation map; it follows from Chevalley's Theorem for Artin stacks \cite[Theorem 5.2]{HallRydh} that $\theta_n(\cC)$ is constructible. Note that the class of $Z_n$ in the Grothedieck ring is given by $\e(Z_n)=(\bL-1)\bL^{n-1}$.

We first show
\[
\cC=\theta_1^{-1}(\theta_1(\cC_1)),
\]
and hence $\cC$ is a cylinder. It is clear that $\cC\subset\theta_1^{-1}(\theta_1(\cC_1))$, so we must show the reverse inclusion. As in the proof of \cite[Proposition 10.10]{SatrianoUsatine} in the paragraph below equation (4), this amounts to showing that if $k'/k$ is a field extension, $\widetilde{\psi}\in\sL(\bA^{r^2})(k')$ and the trunction $\theta_1(\widetilde{\psi})\in Z_1(k')$, then there exists $g\in\sL(\SL_r)(k')$ such that $g\cdot\widetilde{\psi}\in Z(k')$. We may think of $\widetilde{\psi}$ as an $r\times r$ matrix with entries in $k'[[t]]$ and the assumption on $\theta_1(\widetilde{\psi})$ tells us the valuation of the $(i,j)$-entry of $\widetilde{\psi}$ has valuation $1$ for $i=j=1$, valuation $0$ for $1<i=j$, and valuation at least $2$ otherwise. Let $h=\det\widetilde{\psi}$ which has valuation equal to $1$; in particular, $h\neq0$ so there exists a unique element $g\in\SL_r(k'((t)))$ such that $\widetilde{\psi}=g\cdot\diag(h,1,\dots,1)$. Multiplying out the right hand side and comparing valuations, we see the $(i,j)$-entry of $g$ has valuation $0$ for $i=j$, valuation at least $1$ for $i>j=1$, and valuation at least $2$ otherwise. Thus, $g$ has entries in $k'[[t]]$ so defines an element of $\sL(\SL_r)(k')$.

For the remainder of the proof, let $n\geq1$. Let $H_n\subset\sL_n(\SL_r)$ be the closed subgroup scheme whose $A$-valued points for any $k$-algebra $A$, are matrices of the form
\[
\begin{pmatrix}
1 & 0  & \dots & 0\\
a_1t^n & 1 & \dots & 0\\
\vdots & \vdots & \ddots & \vdots\\
a_rt^n & 0 & \dots & 1
\end{pmatrix}
\]
where $a_i\in A[t]/(t^{n+1})$. We have an isomorphism $H_n\simeq\bG_a^{r-1}$ as algebraic groups. In particular, $\e(H_n)=\bL^{r-1}$.

Let $\widetilde{\psi}_n\in Z_n(k')$ for a field extension $k'/k$. If $g_n\in\sL_n(\SL_r)(k')$, we claim that
\[
g_n\cdot\widetilde{\psi}_n\in Z_n(k')\quad \Longrightarrow\quad  g_n\in H_n(k')\quad \Longrightarrow\quad  g_n\cdot\widetilde{\psi}_n=\widetilde{\psi}_n.
\]
Let $\widetilde{\psi}_n=\diag(ut,1,\dots,1)$ where $u\in k'[t]/(t^{n+1})$ is a unit and denote by $(g_n)_{i,j}$ the $(i,j)$-entry of $g_n$. For the first implication, suppose $g_n\cdot\widetilde{\psi}_n=\diag(vt,1,\dots,1)$ with $v\in k'[t]/(t^{n+1})$ a unit. Then for $j>1$, $(g_n)_{i,j}$ is given by the Kronecker delta function $\delta_{i,j}$; for $j>1$, the arc $(g_n)_{1,j}$ annihilates $t$ and hence $(g_n)_{1,j}$ is of the form $a_jt^n$ for $a_j\in k$. Lastly, since $\det(g_n)=1$ we see $(g_n)_{1,1}=1$, so $g_n\in H_n(k')$. This proves the first implication. The second implication is a straightforward check using that elements of the form $at^n$ anihilate $t$.

Then, exactly as in the proof of \cite[Proposition 10.10]{SatrianoUsatine} (see the paragraph where equation (5) is proved), if $\widetilde{C}_n:=\sL_n(\rho)^{-1}(\theta_n(\cC))$, we obtain
\[
\e(\widetilde{C}_n) = \e(Z_n)\e(\sL_n(\SL_r))\e(H_n)^{-1} = (\bL-1)\bL^{n-r}\e(\sL_n(\SL_r)).
\]
Lastly, using Remark 2.4 and Corollary 3.22 of (loc.~cit), 
we have
\[
e(\theta_n(\cC))=e(\widetilde{C}_n)e(\sL_n(\SL_r))^{-1}=(\bL-1)\bL^{n-r}.
\]
It follows that 
\[
\mu_{\sS_r}(\cC)=\lim_{n\to\infty}\e(\theta_n(\cC))\bL^{-(n+1)\dim\sS_r}=(\bL-1)\bL^{-(r+1)}.\qedhere
\]
\end{proof}

\begin{proposition}\label{prop:relative-canonical-Mr}
Let $\cD\subset\sM_r$ denote the preimage of the $B\bG_m$ point under the determinant map $\det\colon\sM_r\to[\bA^1/\bG_m]$. Then
\[
K_{\sM_r/[\bA^1/\bG_m]}=(1-r)\cD.
\]
\end{proposition}
\begin{proof}
By Lemma \ref{l:moduli-framed-bundles}(\ref{l:moduli-framed-bundles::pullback}) and Lemma \ref{l:pullback-K-lci}(\ref{l:pullback-K-lci::pullback}), it suffices to replace $\sM_r$ by $\sS_r$ and $\cD$ by the preimage $\cD'$ of the origin under the map $\det\colon\sS_r\to\bA^1$. By Lemma \ref{l:moduli-framed-bundles}, $\det$ is an isomorphism over $\bA^1\setminus 0$ and so $\cD'$ is the exceptional locus. Letting $\rho\colon\bA^{r^2}\to\sS_r$ be the smooth cover, we know $\rho^*K_{\sS_r/\bA^1}$ is supported on the irreducible divisor $\rho^*\cD'$, hence equals $m\rho^*\cD'$ for some integer $m$; then Proposition \ref{prop:order-fnc-determines-div} proves $K_{\sS_r/\bA^1}=m\cD'$.

Let $\cC$ and $C$ be as in Lemma \ref{l:measures-cylinders-for-mcvf}. Since all arcs in $\cC$ have determinant with valuation $1$, we see $\ord_\cD'=1$ and $\cC\subset|\sL(\sS_r)|\setminus|\sL(\cD')|$. Furthermore, the map $\overline{\cC}(k')\to C(k')$ is bijective by Lemma \ref{l:measures-cylinders-for-mcvf}. Thus, the hypotheses of our motivic change of variables formula Theorem \ref{thm:main-mcvf} are satisfied. Using the computations of the measures given in Lemma \ref{l:measures-cylinders-for-mcvf}, we then have
\begin{align*}
(\bL-1)\bL^{-2}=\mu^{\Gor}_{\bA^1}(C)&=\int_{\cC}\bL^{-\ord_{K_{\sS_r/\bA^1}}}d\mu_\cX\\
&=\bL^{-m}\mu_\cX(\cC)=(\bL-1)\bL^{-(m+r+1)}.
\end{align*}
It follows that $m=1-r$.
\end{proof}

Already Proposition \ref{prop:relative-canonical-Mr} is enough to prove that all $1$-Gorenstein canonical singularities have crepant resolutions by Artin stacks. To construct such resolutions for $\QQ$-Gorenstein log-terminal singularities, we must obtain rational coefficients in our expression for $K_{\sM_r/[\bA^1/\bG_m]}$. For this, we combine the determinant map $\det\colon\sM_r\to[\bA^1/\bG_m]$ with a root stack construction.

In what follows, let $\chi_d\colon[\bA^1/\bG_m]\to[\bA^1/\bG_m]$ be the morphism sending a line bundle with section $(\sL,s)$ to its $d$th power $(\sL^{\otimes d},s^{\otimes d})$. This map is the universal $d$th root stack. To distinguish between $\sM_r$ viewed as an $[\bA^1/\bG_m]$-stack via $\det$ or $\det_d:=\chi_d\circ\det$, we use the notation $K_{\det}$ and $K_{\det_d}$.

\begin{proposition}\label{prop:relative-canonical-Mr-root-stack}
Let $\cD':=B\bG_m\subset[\bA^1/\bG_m]$ denote the closed substack viewed as a Cartier divisor. Then
\[
K_{\det_d}=\left(1-\frac{r}{d}\right){\det}_d^*\cD'
\]
\end{proposition}
\begin{proof}
To distinguish between the two copies of $[\bA^1/\bG_m]$, let $\chi_d$ be the map $\cR_d:=[\bA^1/\bG_m]\to [\bA^1/\bG_m]=:\cR_1$ and let $\cD'_d:=B\bG_m\subset\cR_d$. Since $\chi_d$ is an isomorphism over $\cR_1\setminus\cD'$, Lemma \ref{l:moduli-framed-bundles}(\ref{l:moduli-framed-bundles::birational}) shows that $\det_d$ is as well. By Lemma \ref{l:pullback-K-lci}(\ref{l:pullback-K-lci::composition}), we have
\[
K_{\sM_r/\cR_1}=K_{\sM_r/\cR_d}+{\det}^*K_{\cR_d/\cR_1}.
\]
One computes $K_{\cR_d/\cR_1}=(d-1)\cD'_d=(1-\frac{1}{d})\chi_d^*\cD'$. Note that $\cD$ from Proposition \ref{prop:relative-canonical-Mr} is given by $\det^*\cD'_d=\frac{1}{d}\det_d^*\cD'$. Therefore,
\[
K_{\sM_r/\cR_1}=\frac{1-r}{d}{\det}_d^*\cD' +\left(1-\frac{1}{d}\right){\det}_d^*\cD'=\left(1-\frac{r}{d}\right){\det}_d^*\cD'.\qedhere
\]
\end{proof}

We arrive at the following key theorem. 

\begin{theorem}\label{thm:anti-snc}
Let $Z$ be a smooth irreducible finite type variety and let $D_1\dots,D_n$ be the irreducible components of a simple normal crossing divisor. Let $D=\sum_{i=1}^n m_iD_i$ with $m_i>-1$ rational numbers. Then there exists a good moduli space morphism $\pi\colon\cX\to Z$ where $\cX$ is a smooth irreducible finite type Artin stack with affine diagonal, $\pi$ is an isomorphism over $Z\setminus D$, and there is an equality
\[
-K_{\cX/Z}=\pi^*D
\]
of Cartier divisors.
\end{theorem}
\begin{proof}
Each $D_i$ is cut out by a section $f_i \in \cO(D_i)$; the $(\cO(D_i),f_i)$ define a map $\Phi\colon Z\to[\bA^1/\bG_m]^n$. Since the $D_i$ are the components of a simple normal crossing divisor, $\Phi$ is smooth.
Let $m_i=-1+\frac{r_i}{d_i}$ with $r_i$ and $d_i$ positive integers. Let $\sM:=\sM_{r_1+1}\times\dots\times\sM_{r_n+1}$ and let $p_i\colon\sM\to\sM_{r_i+1}$ be the $i$th projection map. Viewing $\sM_{r_i+1}$ as an $[\bA^1/\bG_m]$-stack via the map $\chi_{d_i}\circ\det$, we have a morphism $\xi\colon\sM\to[\bA^1/\bG_m]^n$. 
We define $\cX$ to be the fiber product
\[
\xymatrix{
\cX\ar[r]^-{\Psi}\ar[d]_-{\pi} & \sM\ar[d]^-\xi\\
Z\ar[r]^-{\Phi} & [\bA^1/\bG_m]^n
}
\]
It follows from Proposition \ref{l:moduli-framed-bundles}(\ref{l:moduli-framed-bundles::birational}) and \cite[Lemma 4.15]{Alper} that $\xi$ is a good moduli space map. Since it is also finite type, we see $\pi$ is a finite type good moduli space map. Next, $\cX$ is smooth since both $\sM$ and $\Phi$ are smooth. By construction, $Z\setminus D$ is the preimage under $\Phi$ of the point $\Spec k=[\bG_m/\bG_m]^n\to [\bA^1/\bG_m]^n$, so another application of Proposition \ref{l:moduli-framed-bundles}(\ref{l:moduli-framed-bundles::birational}) shows $\pi$ an isomorphism over $Z\setminus D$. In particular, irreducibility of $\cX$ follows from that of $Z$. Note also that $\sM$ and $[\bA^1/\bG_m]^n$ have affine diagonal so $\xi$ has affine diagonal; this implies $\pi$ has affine diagonal, and since $Z$ is separated, $\cX$ has affine diagonal.

It remains to compute $K_{\cX/Z}$. 
Let $\cD'_i\subset[\bA^1/\bG_m]^n$ be the preimage of $B\bG_m$ under the $i$th projection map to $[\bA^1/\bG_m]$. Let $\cD_i\subset\sM$ be the inverse image of $\cD'_i$. Applying Proposition \ref{prop:relative-canonical-Mr-root-stack} and Lemma \ref{l:pullback-K-lci}(\ref{l:pullback-K-lci::product}), we see
\[
K_{\sM/[\bA^1/\bG_m]^n}=\sum_{i=1}^n p_i^*K_{\sM_{r_i+1}/[\bA^1/\bG_m]}=\sum_{i=1}^n \left(1-\frac{r_i}{d_i}\right)\cD_i.
\]
Now observe
\[
\pi^*D=\sum_{i=1}^nm_i \pi^*D_i=\sum_{i=1}^nm_i\pi^*\widetilde{\Phi}^*\cD'_i=\sum_{i=1}^nm_i\Psi^*\cD_i=-\Psi^*K_{\sM/[\bA^1/\bG_m]^n}.
\]
Lastly, $\Psi^*K_{\sM/[\bA^1/\bG_m]^n}=K_{\cX/Z}$ by Lemma \ref{l:pullback-K-lci}(\ref{l:pullback-K-lci::pullback}).
\end{proof}

\begin{remark}\label{rmk:anti-snc}
The proof of Theorem \ref{thm:anti-snc} shows that $\cX$ is, in fact, a moduli space. It parameterizes tuples $(\{\cE_i\}_{i=1}^n,\{\beta_i\}_{i=1}^n,\{\iota_i\}_{i=1}^n)$, where $\cE_i$ is a vector bundle of rank $r_i+1$, $\beta_i\colon\cO^{\oplus(r_i+1)}\to\cE_i$ is a morphism, and $\iota_i\colon(\det(\cE_i))^{\otimes d_i}\xrightarrow{\simeq}\cO(D_i)$ is an isomorphism with the property that $\iota_i\circ(\det\beta_i)^{\otimes d_i}$ maps $1$ to $f_i$.
\end{remark}

We may now prove the main theorem of our paper.

\begin{proof}[{Proof of Theorem \ref{thm:crepant-log-terminal}}]
By Hironaka's Theorem there exists a strong resolution of singularities $p\colon Z\to Y$ with exceptional locus given by an effective simple normal crossing  divisor $E$ with irreducible components $D_1,\dots,D_n$. Since $Y$ is $\QQ$-Gorenstein and log-terminal, $K_{Z/Y}=\sum_{i=1}^n m_iD_i$ with $m_i>-1$ rational numbers. By Theorem \ref{thm:anti-snc}, there exists a morphism $\pi\colon\cX\to Z$ where $\cX$ is a smooth irreducible finite type Artin stack with affine diagonal, $\pi$ is an isomorphism on $Z\setminus E=p^{-1}(Y^{\textrm{sm}})$, and $-K_{\cX/Z}=\pi^*K_{Z/Y}$. We see then that $\cX\to Y$ is an isomorphism over $Y^{\textrm{sm}}$ and by Lemma \ref{l:pullback-K-lci}(\ref{l:pullback-K-lci::composition}), 
\[
K_{\cX/Y}=K_{\cX/Z}+\pi^*K_{Y/Z}=0
\]
proving that $\cX\to Y$ is crepant.
\end{proof}

{\section*{Appendix A: Cartier divisors on stacks}}
\renewcommand{\thesection}{A}
\refstepcounter{section}
\label{appendix:Cartier-divs}

While the results in this section are presumably well-known, we include them here for lack of a suitable reference. Throughout this section, $\cX$ denotes a finite type integral Artin stack over an algebraically closed field $k$.\\

\begin{definition}\label{def:Cartier-stacks}
An \emph{effective Cartier divisor} on $\cX$ is a closed substack $\cD\subset\cX$ with invertible ideal sheaf $\cI_\cD$. A \emph{Cartier divisor} is a formal sum $\cD_+-\cD_-$ with $\cD_\pm$ effective Cartier divisors sharing no common components. 
\end{definition}

\begin{definition}\label{def:generic-pt-stack}
If $\cX$ has a dense open subscheme $\cU\subset\cX$, we define the \emph{generic point of} $\cX$ to be the map $\eta\colon\Spec k(\cU)\to\cX$. We use the notation $k(\cX):=k(\cU)$.
\end{definition}

\begin{remark}
The map $\Spec k(\cX)\to\cX$ is independent of the choice of dense open subscheme $\cU$. Indeed, given another choice of dense open subscheme $\cV\subset\cX$, we see $\cU\times_\cX\cV\to\cX$ is also a dense open scheme and we have induced isomorphisms $k(\cU)\xrightarrow{\simeq} K(\cU\times_\cX\cV)\xleftarrow{\simeq} K(\cV)$.
\end{remark}

\begin{proposition}\label{prop:rat-sec-->Cartier}
Suppose $\cX$ has a dense open subscheme and let $\eta\colon\Spec k(\cX)\to\cX$ be the generic point. If $\cL$ is a line bundle and with a choice of isomorphism $\eta^*\cL\simeq\cO_{k(\cX)}$, then there exists a Cartier divisor $\cD=\cD_+-\cD_-$ on $\cX$ such that $\cD_\pm$ are effective Cartier divisors sharing no common components and $\cL=\cI_{\cD_+}^\vee\otimes\cI_{\cD_-}$ as subsheaves of $\eta_*\cO_{k(\cX)}$.

Furthermore, if $j\colon\cU\hookrightarrow\cX$ is an open subscheme and the isomorphism $\eta^*\cL\simeq\cO_{k(\cX)}$ extends to an isomorphism $j^*\cL\xrightarrow{\simeq} \cO_{\cU}$, then $\cD$ is supported on $\cX\setminus\cU$. 
\end{proposition}

\begin{remark}\label{rmk:frac-ideal-vs-rational-section}
A choice of isomorphism $\eta^*\cL\simeq\cO_{k(\cX)}$ is equivalent to a choice of non-zero rational section $s\in\Gamma(\eta^*\cL)$ since $\eta^*\cL$ is a free $\cO_{k(\cX)}$-module of rank one.
\end{remark}

\begin{proof}[{Proof of Proposition \ref{prop:rat-sec-->Cartier}}]
Let $j\colon\cU\hookrightarrow\cX$ be a dense open subscheme, so that $k(\cX)=k(\cU)$ and $\eta\colon \Spec k(\cX)=\Spec k(\cU)\to\cX$. Let $\cK_{\cU}$ be the sheaf of rational functions on $\cU$. Since $\cL$ is torsion-free, the adjunction map $\cL\hookrightarrow\eta_*\eta^*\cL$ is injective. 
We then have an embedding
\[
\iota\colon\cL\hookrightarrow\eta_*\cO_{k(\cU)}=j_*\cK_{\cU}.
\]
Consider the cartesian diagram
\[
\xymatrix{
U''\ar@<-.75ex>[d]_-{p_1} \ar@<.75ex>[d]^-{p_2}\ar@{^{(}->}[r]^-{j''} & X''\ar@<-.75ex>[d]_-{\rho_1} \ar@<.75ex>[d]^-{\rho_2} \\
U'\ar@{^{(}->}[r]^-{j'}\ar[d]_-{p} & X'\ar[d]^-{\rho}\\
\cU\ar@{^{(}->}[r]^-{j} & \cX
}
\]
where $\rho$ is a finite type smooth cover by a scheme and $X'=X'\times_\cX X'$. Each connected component $X'_\ell$ of $X'$ is integral so we may let $\cK_{X'}$ be the sheaf which is constant on $X'_\ell$ with value $K(X'_\ell)$; we may similarly define $\cK_{U'}$. Let $U'_\ell=X'_\ell\cap U'$, let $j'_\ell\colon U'_\ell\hookrightarrow X'_\ell$ be the open immersion.  We then have injections
\begin{equation}\label{eqn:pullbackCartier-fracIdeal}
\rho^*\cL\xhookrightarrow{\rho^*(\iota)}\rho^*j_*\cK_{\cU}\xrightarrow{\simeq} j'_*p^*\cK_{\cU}\hookrightarrow j'_*\cK_{U'}\xleftarrow{\simeq} \cK_{X'}
\end{equation}
where the second map is an isomorphism by flat base change; the third map is an injection since this may be checked component by component, where we may use that $U'_\ell\to\cU$ is a dominant map of integral schemes; the fourth map is an isomorphism since we may again check this component by component where we may use irreducibility of $X'_\ell$ and the fact that $K(X'_\ell)=K(U'_\ell)$.

Thus, on each component, we have an injection $\rho^*\cL|_{X'_\ell}\hookrightarrow \cK_{X'_\ell}$ and hence a Cartier divisor. We have therefore constructed a Cartier divisor $D'=D'_+-D'_-$ on $X'$ such that $D'_\pm$ are effective Cartier divisors sharing no common components and $\rho^*\cL=\cI_{D'_+}^\vee\otimes\cI_{D'_-}$ as subsheaves of $\cK_{X'}$.

Furthermore, since $\rho_1^*\rho^*\iota=\rho_2^*\rho^*\iota$, we see that the injection $\delta\colon\rho^*\cL\hookrightarrow \cK_{X'}$ constructed in \eqref{eqn:pullbackCartier-fracIdeal} satisfies $\rho_1^*\delta=\rho_2^*\delta$. As a result, $\rho_1^*D'=\rho_2^*D'$ as Cartier divisors. Note that for each $i$, $\rho_i^*D'_+$ and $\rho_i^*D'_-$ do not share a common component; indeed, if they did then there would be an irreducible component $Z_i\subset\rho_i^{-1}D'_+$ such that $\rho_i(Z)$ has codimension strictly greater than $1$, i.e., the generic point of $Z_i$ does not map to the generic point of a component of $D'_+$, which contradicts flatness of the map $\rho_i\times_{X'}D'_+\colon\rho_i^{-1}D'_+\to D'_+$. Thus, the equality $\rho_1^*D'=\rho_2^*D'$ implies $\rho_1^*D'_\pm=\rho_2^*D'_\pm$ as effective Cartier divisors, i.e., we have equality of the ideal sheaves $\rho_1^*\cI_{D'_\pm}=\rho_2^*\cI_{D'_\pm}$ as subsheaves of $\cO_{X''}$. By descent, we then obtain closed substacks $\cD_\pm\subset\cX$; the ideal sheaves $\cI_{\cD_\pm}$ are invertible since this can be checked smooth locally where we have $\rho^*\cI_{\cD_\pm}=\cI_{D'_\pm}$. Lastly, we have a natural embedding of $\cI_{\cD_+}^\vee\otimes\cI_{\cD_-}$ into $j_*\cK_\cU$ whose image agrees with $\cL$ since this may also be checked locally, where $\rho^*\cL=\cI_{D'_+}^\vee\otimes\cI_{D'_-}$ in $\cK_{X'}$.

If the isomorphism $\eta^*\cL\simeq\cO_{k(\cX)}$ extends to an isomorphism $j^*\cL\xrightarrow{\simeq} \cO_{\cU}$ then $j^*\cL=\cO_\cU$ as subsheaves of $\cO_{k(\cX)}$, and hence $\cD|_\cU=0$ by definition, i.e., $\cD$ is supported on $\cX\setminus\cU$.
\end{proof}


\end{document}